\documentclass[a4paper,11pt]{amsart} 
\usepackage{amssymb, amsmath}
\usepackage{amscd}
\numberwithin{equation}{section}
\usepackage{epsfig}
\usepackage{amsmath}
\usepackage{amsfonts,amssymb,amsopn}

\usepackage{amsthm}
\usepackage{hyperref}
\usepackage{verbatim}
\usepackage{color}
\usepackage[color,all]{xy}
\usepackage{xcolor}

\newcommand{\cE}{{\mathcal E}}
\newcommand{\cF}{{\mathcal F}}
\newcommand{\cH}{{\mathcal H}}
\newcommand{\cL}{{\mathcal L}}

\newcommand{\cO}{{\mathcal O}}
\newcommand{\cU}{{\mathcal U}}

\newcommand{\cX}{{\mathcal X}}

\newcommand{\bB}{{\mathbb B}}
\newcommand{\C}{{\mathbb C}}
\newcommand{\bH}{{\mathbb H}}
\newcommand{\bN}{{\mathbb N}}
\newcommand{\PP}{{\mathbb P}}
\newcommand{\bZ}{{\mathbb Z}}

\newcommand{\tT}{\widetilde{T}}

\newcommand{\fg}{\mathfrak{g}}
\newcommand{\fh}{\mathfrak{h}}

\newcommand{\fp}{\mathfrak{p}}
\newcommand{\fq}{\mathfrak{q}}
\newcommand{\fs}{\mathfrak{s}}

\newcommand{\Sym}{\mathrm{Sym}}
\newcommand{\Ker}{\mathrm{Ker}}
\newcommand{\Aut}{\mathrm{Aut}\,}
\newcommand{\Image}{\mathrm{Im}\,}
\newcommand{\Iden}{\mathrm{Id}}
\newcommand{\Hom}{\mathrm{Hom}}
\newcommand{\rank}{\mathrm{rk}\,}
\newcommand{\isom}{\xrightarrow{\sim}}

\newcommand{\Kx}{K_{X}}
\newcommand{\Ox}{{{\cO}_{X}}}

\newcommand{\Pic}{\mathrm{Pic}}
\newcommand{\Sec}{\mathrm{Sec}}
\newcommand{\Hilb}{\mathrm{Hilb}}

\newcommand{\GL}{\mathrm{GL}}
\newcommand{\PGL}{\mathrm{PGL}}
\newcommand{\SL}{\mathrm{SL}}
\newcommand{\SO}{\mathrm{SO}}

\newcommand{\Sp}{\mathrm{Sp}}

\newcommand{\PSp}{\mathrm{PSp}}
\newcommand{\Spin}{\mathrm{Spin}}
\newcommand{\Or}{\mathrm{O}}
\newcommand{\Gr}{\mathrm{Gr}}
\newcommand{\OG}{\mathrm{OG}}

\newcommand{\Urd}{U_X (r, d)}

\newcommand{\MGd}{M_X (G, \delta)}
\newcommand{\Ms}{M_X (G, \delta ; P, s)}

\newcommand{\ud}{{\underline{d}}}
\newcommand{\uU}{{\underline{U}}}
\newcommand{\Ru}{{\mathrm{R_u}}}
\newcommand{\Mso}{\Ms^\circ}

\newtheorem{theorem}{{\textbf Theorem}}[section]
\newtheorem{proposition}[theorem]{{\textbf Proposition}}
\newtheorem{corollary}[theorem]{{\textbf Corollary}}
\newtheorem{lemma}[theorem]{{\textbf Lemma}}

\newtheorem{remit}[theorem]{{\textbf Remark}}
\newenvironment{remark}{\begin{remit}\rm}{\end{remit}}
\newtheorem{exit}[theorem]{{\textbf Example}}
\newenvironment{example}{\begin{exit}\rm}{\end{exit}}
\newtheorem{defnit}[theorem]{{\textbf Definition}}
\newenvironment{definition}{\begin{defnit}\rm}{\end{defnit}}

\newtheorem{mainthm}{Theorem}

\title[Segre invariants of principal bundles]{Segre invariants of principal bundles over a curve}

\author{George H.\ Hitching}
\address{G.\ H.\ Hitching, Oslo Metropolitan University, Postboks 4, St. Olavs plass, 0130 Oslo, Norway} \email{gehahi@oslomet.no}

\author{Alfonso Zamora}

\address{A. Zamora, Departamento de Matem\'atica Aplicada a las TIC, ETSI Inform\'aticos, Universidad Polit\'ecnica de Madrid, Campus de Montegancedo, 28660 Madrid, Spain} \email{alfonso.zamora@upm.es}

\begin{document}

\begin{abstract}
For a vector bundle $V$ over a curve $X$, the Segre invariant $s_n (V)$ encodes the maximal degree attained by rank $n$ subbundles of $V$. The functions $s_n$ define stratifications on moduli of $V$ which are well studied. Let $G$ be a connected reductive algebraic group, and $E \to X$ a principal $G$-bundle. For each parabolic subgroup $P \subset G$ there is a Segre number $s_P (E)$, generalising $s_n (V)$. We show that $s_P$ is semicontinuous in families of $G$-bundles, and thus defines stratifications on moduli spaces of $G$-bundles over $X$. We study the invariance properties of $s_P$, relating the behaviour of $s_P$ and $s_{\phi(P)}$ for a surjective homomorphism $\phi \colon G \to H$ and allowing us to compare the Segre stratifications for $G$ and $H$. Finally, we analyse the stratification for the Borel subgroup $B$ of $\GL_3$, identifying patterns in the geometry and proving, in particular, a sharp Hirschowitz-type bound on $s_B (E)$ for certain topological types.
\end{abstract}

\maketitle

\section{Introduction}

Let $X$ be a complex projective smooth curve of genus $g \ge 2$. A vector bundle $V$ of rank $r$ and degree $d$ over $X$ is said to be \emph{semistable} if
\[
\mu(W)=\frac{\deg W}{\rank W} \ \le \ \mu(V)=\frac{d}{r}
\]
for all proper nonzero subbundles $W \subset V$. It is \emph{stable} if inequality is strict for all $W$. The relevance of stability for moduli questions and other purposes is very well known: a moduli space $\Urd$ exists for stable vector bundles of rank $r$ and degree $d$ over $X$, which is an irreducible quasiprojective variety of dimension $r^2 (g-1) + 1$. More generally, each vector bundle over $X$ admits a unique Harder--Narasimhan filtration \cite{HN75}, and the topological invariants of this filtration (ranks and degrees) are shown to be semicontinuous \cite{Sha77} and to define a stratification on any family of bundles \cite{Nit11}.

A refinement of the notion of stability is the following. For each $1 \le n \le r-1$, the \emph{Segre invariant} $s_n (V)$ is defined by
\begin{equation} \label{sndefn}
s_n (V) \ := \ \min \{ nd - r \deg W : W \subset V \hbox{ a subbundle of rank } n \} .
\end{equation}
Noting that $nd - r \deg W = rn ( \mu (V) - \mu(W) )$, we see that $V$ is stable if and only if $s_n (V) > 0$ for each $n$, and strictly semistable if $s_n (V) \geq 0$ for each $n$ and there exists a $k$ such that $s_k(V)=0$. Thus Segre invariants are naturally ``degrees of stability'', measuring how far a vector bundle is from being destabilised by a rank $n$ subbundle. The term ``invariant'' is applied to $s_n$ because for any $V$, for $1 \le n \le \rank V$ we have $s_n (V \otimes L) = s_n (V)$ for all line bundles $L$. In other words, $s_n$ is determined by the projective bundle $\PP V$. 

We recall that stability and semistability are open conditions on families of bundles. Generalising this statement we get: if $s \equiv nd \;(\!\!\mod r)$ then the locus
\begin{equation} \label{Segrestratum-vectorbundles}
U_X (r, d; n, s) \ = \ \{ V \in \Urd : s_n (V) = s \}
\end{equation}
is locally closed in $\Urd$, and $U_X (r, d; n, s')$ is contained in the closure of $U_X (r, d; n, s)$ for $s' < s$. Thus the $s_n$ define stratifications on $\Urd$, which are well understood \cite{RT, BL}. 

Notice that there are important differences between Segre and Harder--Narasimhan stratifications. As mentioned above, the Harder--Narasimhan filtration of any bundle $V$ is unique. In contrast, if $V$ is semistable then a proper subbundle of maximal slope and maximal rank may not be unique (consider a direct sum of line bundles of the same degree). This is reflected in the fact that the Jordan--H\"older filtration of a semistable bundle is not in general unique; but only the associated graded bundle. If $V$ is stable, then it may have many maximal subbundles of a given rank with mutually nonisomorphic graded bundles. (Indeed, the enumeration of maximal subbundles is a focus of \cite{Hol04} and \cite{LN}). 

Furthermore, although for $r = 2$ the Segre and Harder--Narasimhan stratifications on any family of bundles coincide, for $r \ge 3$ they diverge. There is a single Harder-Narasimhan stratification on any family whose terms have different ranks and degrees. In contrast, the family admits a distinct Segre stratification for \emph{each} rank $n$ with $1 \le n \le r-1$, and the stratum to which a given bundle $V$ belongs is determined solely by the maximal degree of subbundles of rank $n$ in $V$.

The goal of the present work is to study a generalisation of the Segre invariant for an arbitrary connected reductive group $G$ and an arbitrary parabolic subgroup $P$. We begin by recalling the definition of stability for principal $G$-bundles. 

\begin{definition}\cite[Definition 1.1]{Rth1} \label{stabilityGbundle}
Let $G$ be a connected reductive linear algebraic group over $\C$. A principal $G$-bundle $E \to X$ is said to be \emph{semistable} if for every maximal parabolic subgroup $Q \subset G$ and every reduction of structure group $\sigma \colon X \to E/Q$, we have $\deg \sigma^* T_{\pi} \geq 0$, where $T_{\pi}$ is the tangent bundle along fibres of the the associated $G/Q$-fibration $\pi: E/Q \to X$. A principal $G$-bundle $E \to X$ is said to be \emph{stable} if strict inequality holds for all $\sigma$.
\end{definition}

By \cite{Rth2}, for each $\delta \in \pi_1 (G)$ there exists a moduli space $\MGd$ parameterising isomorphism classes of stable $G$-principal bundles of topological type $\delta$. This is an irreducible quasiprojective variety of dimension $(g-1) \cdot \dim G + \dim Z(G)$. Following \cite{HN} and \cite{CH1}, we consider a function which refines the above notion of stability in the same way as the Segre invariants $s_n$ refine stability for vector bundles.

\begin{definition} \label{sPDefn}
Let $E \to X$ be a principal $G$-bundle. Let $P$ be a parabolic subgroup, not necessarily maximal, and $\pi \colon E/P \to X$ the associated $G/P$-fibration. We define
\[
s_P (E) \ := \ \min \left\{ \deg \sigma^* T_\pi : \sigma \colon X \to E/P \hbox{ a reduction of structure group to } P \right\} .
\]
\end{definition}

\noindent We will see in {\S} \ref{characterisations} that this is equivalent to the function denoted $s(V; P)$ in \cite[p.\ 193]{CH1}. In view of Definition \ref{stabilityGbundle}, we see that when $P$ ranges over all maximal parabolics of $G$, the numbers $s_P (E)$ of Definition \ref{sPDefn} give a measure of the stability of $E$. Note that $E$ is stable if $s_P(E)>0$ for every maximal parabolic $P$, and strictly semistable if $s_P(E)\geq 0$ for every maximal parabolic $P$ and there exists a maximal parabolic $Q$ for which $s_Q(E)=0$.

Let us show that $s_P$ generalises the Segre invariants $s_n$ defined in (\ref{sndefn}). It is well known that any rank $r$ vector bundle $V$ is an associated bundle of a principal $\GL_r$-bundle $E \to X$, and that a rank $n$ subbundle $F \subset V$ corresponds naturally to a reduction of structure group $\sigma \colon X \to E/P$ where $P \subset \GL_r$ stabilises a subspace of dimension $n$ in $\C^r$. We observe that $s_P (E)$ coincides with the Segre invariant $s_n (V)$ defined above. For: In this case, $\pi \colon E/P \to X$ is canonically identified with the Grassmann bundle $\Gr (n, V)$, and $F \cong \sigma^* \cU$ where $\cU \subset \pi^* V$ is the universal subbundle. As $T_\pi \cong \Hom ( \cU, \pi^* V / \cU )$, we have $\sigma^* T_{\pi} \cong \Hom ( F, V/F )$ and
\begin{equation} \label{sPsndef}
\deg \sigma^* T_\pi \ = \ n \cdot ( \deg V - \deg F) - (r - n) \cdot \deg F \ = \ n \cdot \deg V - r \cdot \deg F .
\end{equation}
It now follows from (\ref{sndefn}) that $s_P (E) = s_n (V)$. Thus $s_P$ is indeed a generalisation of the Segre invariant $s_n$ for vector bundles.

For another example: Suppose that $W$ is a \emph{symplectic bundle}; that is, $W$ admits an $\Ox$-valued nondegenerate skewsymmetric bilinear form. By linear algebra, $\rank W$ is an even number $2n$, and an isotropic subbundle $F \subset W$ has rank at most $n$. In \cite{CH1, CH2}, a Segre-type invariant is defined which encodes the maximal degree attained by rank $n$ isotropic subbundles of a symplectic bundle:
\begin{equation}\label{sn-symp-def}
t ( W ) \ := \ \min \{ -2 \cdot \deg F : \, F \hbox{ a rank $n$ isotropic subbundle of } W \}.
\end{equation}
Now in this case, $W$ is the associated bundle of a principal $\Sp_{2n}$-bundle $E \to X$, and a rank $n$ isotropic subbundle corresponds to a reduction $\sigma \colon X \to E/P$ where $P \subset \Sp_{2n}$ is the parabolic subgroup stabilising an isotropic subspace $\Lambda \subset \C^{2n}$ of dimension $n$. Here $E/P$ is identified with the \emph{Lagrangian Grassmann bundle} parameterising $n$-dimensional isotropic subspaces of fibres of $W$. By Lie theory, $\fs\fp_{2n} / \fp \cong \Sym^2 \Lambda^*$, whence $T_\pi \cong \Sym^2 \cU^*$ where $\cU \subset \pi^* W$ is the universal isotropic subbundle. As $F = \sigma^* \cU$, we have $\sigma^* T_\pi \cong \Sym^2 F^*$. 
It follows that
\[
s_P (E) \ = \ \min \{ \deg \,\Sym^2 F^* : F \subset W \hbox{ a rank $n$ isotropic subbundle} \}
\]
and the number (\ref{sn-symp-def}) satisfies $t(V) = \frac{2}{n+1} \cdot s_P (E)$. Thus Definition \ref{sPDefn}, up to a constant factor\footnote{The choice of coefficient $-2$ rather than $-(n+1)$ in (\ref{sn-symp-def}) was to make $t(V)$ a generalisation of the Segre invariant of rank two bundles.}, also generalises the invariant $t(V)$.

The invariant $t(V)$ also makes sense for maximal rank isotropic subbundles of \emph{orthogonal} vector bundles, and is studied in \cite{CH2, CH3}. If $G = \SO_r$ and $P$ is the stabiliser of an isotropic subspace of maximal dimension $\left\lfloor \frac{r}{2} \right\rfloor$ in $\C^r$, a similar argument to the above shows that $s_P$ also generalises $t(V)$ in the orthogonal case, up to a constant factor.

The structure and results of the present article are as follows. In section \ref{sec:G} we collect preliminary results on reductive groups and principal $G$ bundles. In section \ref{FirstProps}, we give some alternative definitions of and some elementary properties of the function $s_P$, and prove:

\begin{mainthm}\label{thm:mainA}
Let $G$ be a connected reductive algebraic group and let $P\subset G$ be a parabolic subgroup. 
\begin{itemize}
\item[(i)] The function $s_P$ is semicontinuous in families of $G$-bundles of fixed topological type $\delta \in \pi_1 (G)$. Thus it induces a stratification of the moduli space $M_X(G,\delta)$ into locally closed subsets 
\[\Mso \ := \ \{ E \in \MGd : s_P (E) = s \}.\] 
\item[(ii)] The function $s_P$ depends only on the conjugacy class of $P\subset G$.
\end{itemize}
\end{mainthm}
The content of Theorem \ref{thm:mainA} (i) is proved in Theorem \ref{semicont} and Corollary \ref{cor-stratification}. We then proceed in section \ref{invariance} to study invariance properties of $s_P$ generalising the invariance properties of the Segre invariant $s_n$ for vector bundles. We determine the relation between $s_P$ and $s_{\phi (P)}$ where $\phi \colon G \to H$ is a surjective homomorphism of connected algebraic groups (Proposition \ref{sPinv}). This will justify the term ``invariant'' and allow us to prove Theorem \ref{thm:mainA} (ii) in Corollary \ref{ConjInv}.

This allows us in section \ref{sec:stratificacions} to compare the stratifications on certain components of the moduli spaces of principal bundles for two different groups related by a surjective homomorphism. 
\begin{mainthm} \label{thm:mainB}
Let $\phi \colon G \to H$ be a surjective homomorphism of connected reductive algebraic groups such that $\Ker (\phi)_0 \subseteq Z(G)$. 
\begin{enumerate}
\item[(i)] For every parabolic subgroup $P \subset G$, the Segre invariants $s_P (E)$ and $s_{\phi(P)} (E(H))$ coincide. In particular, a principal $G$-bundle $E \to X$ is (semi)stable if and only if the associated principal $H$-bundle $E(H)$ is (semi)stable.
\item[(ii)] Given topological types $\delta \in \pi_1 (G)$ and $\phi_*\delta\in\pi_1(H)$, we obtain a morphism between the moduli spaces $\phi_\delta \colon M_X (G, \delta ) \ \to \ M_X (H, \phi_* \delta )$ induced by extension of structure group which relates the Segre stratifications by 
\[\phi_\delta^{-1} M_X (H, \phi_* \delta ; \phi (P), s)^\circ = M_X (G, \delta ; P, s)^\circ.\] 
Moreover, whenever $\phi_\delta$ is surjective, the number of non-empty strata $\Ms^\circ$ equals the number of non-empty strata  $M_X (H, \phi_* \delta ; \phi (P), s)^\circ$.
\end{enumerate}
\end{mainthm}

Theorem \ref{thm:mainB} is proved in the statements of Theorem \ref{RelateStrata}. Examples are given in $\S$ \ref{sec:examples}: by using the adjoint map and the universal covering map we relate stratifications of $\GL_r$ and $\SL_r$ to $\PGL_r$ (and intermediate quotients), symplectic and projective symplectic groups, and orthogonal and spin groups, covering all substantial cases of complex classical Lie groups of Dynkin type ABCD, from the simply connected to the adjoint group. 

The examples in section \ref{sec:stratificacions} show that the Segre function remains invariant when taking the adjoint group of $G$ and seem to indicate that the information on the Segre stratification is intrinsically contained the moduli space of adjoint bundles, as happens with $\GL_r$ and $\PGL_r$. One might thus suppose that the center of $G$ does not substantially affect the information contained in Segre invariant: we obtain a correspondence between the stratifications for the components of trivial topological type in the corresponding moduli spaces of $G$ and $G$ mod a subgroup of the center of $G$. What happens with the other components of the moduli space of $G$-bundles, the behaviour of the stratification under Langlands dual groups (which exchange the fundamental group and the centre) and the geometry of the stratification for other groups are interesting questions which will be the subject of further investigation. 

We finally devote section \ref{sec:Borel} to studying the simplest case not corresponding to a maximal parabolic subgroup: the stratification associated to the Borel subgroup $B$ of $\GL_3$.  Contrasting with the case of a maximal parabolic, the strata are no longer irreducible, and in fact are not equidimensional, in general. 

For each $\ud := (d_1, d_2, d_3) \in \bZ^3$ with $d_1+d_2+d_3=\delta$, we denote by $M ( \ud )$ the locus of stable principal $\GL_3$-bundles in $M_X (\GL_3, \delta)$ admitting a $B$-reduction of type $\ud$ (see section \ref{ssec:iterated_extensions}). The following statements are proven in Section \ref{sec:Borel}.

\begin{mainthm}\label{thm:mainC}
Let $\ud=(d_1, d_2, d_3)\in\bZ^3$ be a tuple satisfying $d_1 < \frac{\delta}{3} < d_3$. Set $s = 2 (d_3 - d_1 )$.

\begin{itemize}
\item[(i)] The locus $M(\ud)$ is nonempty and irreducible of dimension at most $6g - 5 + s$. Moreover, $M(\ud)$ is contained in the Segre stratum $M_X (\GL_3, \delta; B, s)$.
\item[(ii)] We have
\[M (\ud) \ \subseteq \ \overline{M (d_1 - 1, d_2 + 1, d_3)} \cap \overline{M (d_1, d_2 - 1, d_3 + 1)}.\]
In particular, $M_X (\GL_3, \delta ; B, s)^\circ \subseteq \overline{M_X (\GL_3, \delta ; B, s + 2)}$.
\item[(iii)] (Hirschowitz bound) For $\delta \equiv 0 \;(\!\!\mod 3)$ and $g \equiv 1 \;(\!\!\mod 6)$, any $E \in M_X (\GL_3, 0)$ satisfies $s_B (E) \le 3 (g-1)$.
\end{itemize}
\end{mainthm}

For certain choices of $\ud$, we can show that the dimension bound (i) on $M (\ud)$ is sharp (Proposition \ref{prop:GenFin}); while in other cases, $M(\ud)$ may be contained in $M_X (G, \delta; B, s')$ for $s' < s$ (Remark \ref{BadBeh}). 

A word of explanation regarding (iii) above: The number $s_P (E)$ can be arbitrarily large negative (for unstable bundles). However, \cite[Theorem 1.1]{HN} shows that it is bounded above $\deg \sigma^* T_\pi \ \le \ \dim g \cdot (G/P)$. For $G = \GL_r$ and $P$ a maximal parabolic, Hirschowitz \cite{Hir} gave a sharpening of this bound (see also \cite{CH0}). Analogous sharp bounds on $s_P$ for parabolics in the groups $\Sp_{2n}$ and $\SO_r$ stabilising maximal isotropic subspaces appeared in \cite{CH2, CH3}. Corollary \ref{cor:Hirschowitz} proves Theorem \ref{thm:mainC} (iv), giving a sharp upper bound for $s_B$; to our knowledge, the first example for a nonmaximal parabolic.

\subsection*{Acknowledgements}

The first named author thanks Insong Choe for very many stimulating conversations on Segre stratifications, and for his interest in this project. The second named author thanks Oslo Metropolitan University for hospitality while this work was carried out, and is partially supported by project PID2022-142024NB-I00 by the Spanish government.

\subsection*{Notation}

We denote the Lie algebra of an algebraic group ($G, H, P, \ldots$) by the corresponding Fraktur letter $(\fg, \fh, \fp, \ldots)$ If $\pi \colon F \to Y$ is a smooth morphism, we write
\[
T_\pi \ := \ \Ker \left( T_F \ \to \ T_Y \right) ,
\]
for the vertical tangent bundle; that is, is the subbundle of $T_F$ of vectors tangent to the fibres of $\pi$.

\section{Reductive groups and principal \texorpdfstring{$G$}{G}-bundles}\label{sec:G}

In this section, we review necessary facts on reductive algebraic groups and on principal $G$-bundles and their classification.

\subsection{Reductive groups}

A reference for the following is \cite{Spr}.

Let $G$ be a connected algebraic group over $\mathbb{C}$. The \emph{unipotent radical} $\Ru G$ is the maximal closed connected unipotent subgroup of $G$. If $\Ru G$ is trivial, then $G$ is said to be \emph{reductive}. A \emph{Borel} subgroup $B \subseteq G$ is a maximal closed and connected solvable algebraic subgroup. Fixing a faithful representation of $G$ on a vector space $V$, the quotient $G/B$ can be identified with the space of (generalized) full flags on $V$ and called the \emph{full-flag variety}; it is a complete variety. A subgroup $P \subseteq G$ is said to be \emph{parabolic} if $G/P$ is a complete variety or equivalently if $P$ contains a Borel subgroup. A Borel subgroup $B$ is a minimal parabolic subgroup with respect to inclusion.

A \emph{maximal torus} $T \subseteq G$ is a connected abelian subgroup that is maximal with respect to the inclusion, and any two maximal tori are conjugate by the action of $G$. We can fix a maximal torus $T$ and a Borel subgroup $B$ such that $T \subseteq B \subseteq G$, and then each conjugacy class of parabolic subgroups contains a unique representative $P$ satisfying $T \subseteq B \subseteq P \subseteq G$.

Let $X^{\ast}(T)=\{\chi:T\rightarrow \C^*\}$ and  $X_{\ast}(T)=\{\lambda:\C^*\rightarrow T\}$ be the \emph{lattices of characters} and \emph{cocharacters} (or $1$-parameter subgroups) of $T$. These are dual free abelian groups via the $\bZ$-pairing
\begin{equation} \label{XXpairing}
\left\langle \chi, \lambda \right\rangle \ \mapsto \ \deg \left( t \mapsto \chi ( \lambda (t) ) \right)
\end{equation}
for $\chi \in X^\ast(T)$ and $\lambda \in X_\ast(T)$.

Given a parabolic subgroup $P \subset G$, we have an exact sequence
\begin{equation}\label{UPLsequence}
1 \ \to \ \Ru P \ \to \ P \ \to \ L \ \to \ 1
\end{equation}
where $L$ is the maximal reductive quotient of $P$. Then $P$ contains $L$ as a \emph{Levi subgroup}, and we have the \emph{Levi decomposition} $P=\Ru P \rtimes L$. There are no nontrivial algebraic homomorphisms $\Ru P \to \C^*$, and therefore the characters of $P$ and $L$ are the same:
\begin{equation} \label{XstarP}
X^* (P) \ \cong \ X^* (L) .
\end{equation}

\subsection{Principal \texorpdfstring{$G$}{G}-bundles} \label{sec:principalGbundles}

Let us now recall basic facts about principal $G$-bundles. References for the following include \cite{Bal} and \cite{Gro}.

Let $Y$ be a variety over $\C$, and let $G$ be a complex algebraic group. A \emph{principal $G$-bundle} over $Y$ is a variety $E$ admitting a free right action of $G$ with quotient $Y$, and such that $E \to Y$ is locally trivial in the \'etale topology. In the same way as a vector bundle, a principal bundle $E$ may be recovered up to isomorphism from a cocycle of transition functions $g_{\alpha \beta} \colon U_{\alpha} \cap U_{\beta} \ \to \ G$ over a suitable open cover $\{ U_\alpha \}$ of $Y$. The cocycle $\{ g_{\alpha \beta} \}$ determines an element of the nonabelian cohomology set $H^1 (X, G)$.

\subsubsection{Associated bundles} \label{AssocBundles}

Let $F$ be a variety admitting a left action of $G$ via an algebraic representation $\rho \colon G \to \Aut (F)$. Given a $G$-bundle $E \to Y$, one may form the \emph{associated bundle}
\begin{equation} \label{AssocBundle}
E (\rho, F) \ := \ \frac{E \times F}{(e, f) \sim ( e \cdot g, \rho(g^{-1}) \cdot f)}
\end{equation}
which is a fibre bundle over $Y$ with fibre $F$. If no confusion should arise, we may write $E(\rho)$ or $E(F)$ for $E(\rho, F)$.

We list some important associated bundles of a principal $G$-bundle $E \to Y$, which will be used throughout.

\begin{itemize}
\item (\emph{Associated vector bundles}) If $\rho \colon G \to \GL_r$ is a linear representation, then the associated bundle $E(\rho, \C^r)$ is a vector bundle of rank $r$. In particular, if $\chi \colon G \to \C^* = \GL_1$ is a character, then $E(\chi, \C)$ is a line bundle.

\item (\emph{Extension of structure group}) Given a group homomorphism $G \to H$, we have an action of $G$ on $H$ by left translations. Then the associated bundle $E ( H )$ is in fact a principal $H$-bundle via the right $H$-action
\[
(e, h) \cdot \ell \ = \ (e, h \ell)
\]
for $(e, h)$ an equivalence class as in (\ref{AssocBundle}). In this case $E(H)$ is said to be obtained from $E$ by \emph{extension of structure group}.

\item (\emph{Homogeneous space bundles}) Let $Q \subset G$ be a closed subgroup. Consider the action of $G$ given by left translations of the left cosets $gQ$ for $g \in G$. Then the associated bundle $E(G/Q)$ is a fibre bundle with fibre $G/Q$, which we will often denote by $\pi \colon E/Q \rightarrow Y$. It is naturally identified with the globalised coset space $\{ eQ : e \in E \}$.
\end{itemize}

\subsubsection{Reductions of structure group} \label{sec:reductions}

Let $E \to Y$ be a principal $G$-bundle. If $Q \subset G$ is a closed subgroup, then $E \to E/Q$ is a principal $Q$-bundle. Let $\pi \colon E/Q \to Y$ be the projection. A section $\sigma \colon Y \to E/Q$ of $\pi$ determines a $Q$-bundle $\sigma^* E$ over $Y$, which will be denoted by $E_{Q, \sigma}$, or $E_Q$ when $\sigma$ is understood from the context. In local trivialisations, $\sigma$ may be understood as prescribing a choice of coset in $G/Q$ for each $y \in Y$.

Let us recall some key examples of reduction of structure group. 

\begin{example} \label{example:reductions}
Let $G$ be a connected reductive group and $E \to X$ a principal $G$-bundle.

\begin{enumerate}
\renewcommand{\labelenumi}{(\alph{enumi})}
\item Here we take $G = \GL_r$ or $\SL_r$. If $Q \subset G$ is a maximal parabolic subgroup, then $Q$ is the stabiliser of some subspace $\C^n \subset \C^r$ and $G/Q$ is the Grassmannian $\Gr (n, r)$. A reduction of structure group $\sigma \colon Y \to E/Q$ is equivalent to a choice of rank $n$ subbundle of $E(\C^r)$.
\item More generally, if $P \subset G$ is an arbitrary parabolic subgroup, then $P$ is obtained as the stabiliser of some (partial) flag $0 \subset \mathsf{V}_1 \subset \cdots \subset \mathsf{V}_{t-1} \subset \mathsf{V}_t = \C^r$ where $t \le r$. Then a reduction $\sigma \colon X \to E/P$ is equivalent to a flag of subbundles $0 \subset F_1 \subset \cdots \subset F_{t-1} \subset F_t = E(\C^r)$, where $\rank F_i = \dim \mathsf{V}_i$ for $1 \le i \le t$.
\item If $G = \Sp_{2n}$, a parabolic $P$ is the stabiliser of a (partial) flag of isotropic subspaces, and a reduction to $P$ is equivalent to a flag of isotropic subbundles of the associated vector bundle $E (\C^{2n})$. A similar statement holds for $G = \SO_r$.
\end{enumerate}

\end{example}

\subsection{Classification of \texorpdfstring{$G$}{G}-bundles}

Let us recall the correspondence between topological $G$-bundles and elements of the fundamental group of $G$. 
\begin{proposition} \cite[Proposition 5.1 and Remark 5.1]{Rth1} Let $G$ be a connected linear algebraic group.
\begin{enumerate}
\renewcommand{\labelenumi}{(\alph{enumi})}
\item Topological types of $G$-bundles on $X$ are in natural bijection with elements of $\pi_1 (G)$.
\item Let $E \to X$ be an $G$-bundle of topological type $\delta \in \pi_1 (G)$. If $\phi \colon G \to H$ is a homomorphism, then $E(H)$ is of topological type $\phi_* \delta \in \pi_1 (H)$, where $\phi_*$ is the induced map between the fundamental groups.
\end{enumerate}
\end{proposition}

We now discuss the notion of degree for principal bundles. This is a slight generalisation of \cite[Definition 3.2]{HN}, and is essentially equivalent to the notion of numerical type in \cite{Hol04}. See also \cite[{\S} 2.2.2]{Sch}.

\begin{definition} \label{defn:degree}
Let $H$ be any algebraic group, not necessarily reductive, and let $E$ be an $H$-bundle. Given a character $\chi \in X^* (H)$, we can form the associated line bundle $E(\chi, \C)$. Then the map
\[
\chi \ \mapsto \ \deg E(\chi, \C)
\]
defines an element $d_H (E)$ of $\Hom ( X^* (H), \bZ )$, which we call the \emph{degree} of $E$.
\end{definition}

For $G$ reductive, $d_G (E)$ is exactly the degree $d_E$ of \cite[Definition 3.2]{HN}. In particular, if $G = \GL_r$ then $d_G (E)$ coincides with the usual degree of the vector bundle $E (\C^r)$. To see this, observe that the determinant generates $X^* ( \GL_r )$, and evaluation at the determinant defines an isomorphism $\Hom ( X^* (\GL_r), \bZ ) \isom \bZ$ which sends $d_{\GL_r} (E)$ to $\deg E( \C^r )$. 

\begin{definition} \cite[p.\ 1]{Hol02} \label{def_num_type}
Let $P \subset G$ be a parabolic subgroup and and $\sigma \colon X \to E/P$ a reduction of structure group. Then the \emph{numerical type} of $\sigma$, denoted by $[ \sigma ]$, is the degree $d_P (E_{P, \sigma})$.
\end{definition}

\begin{remark}\label{rem-degreeGLn}
We include the group $H$ in the notation $d_H (E)$ for clarity. Let $\phi \colon G \to H$ be any homomorphism of algebraic groups. Then there is a natural map
\[
\phi_* \colon \Hom ( X^* (G) , \bZ) \ \to \ \Hom (X^* (H), \bZ )
\]
given by restriction to the image of $\phi^* \colon X^* (H) \to X^* (G)$. Unwinding definitions, we see that $d_H ( E(H) ) = \phi_* d_G ( E )$.

We observe this because in what follows, we may wish to consider the degree $d_G (E) \in \Hom (X^*(G), \bZ)$ or the degree $d_P ( E_P )$ of a $P$-bundle obtained from $E$ by reduction of structure group; and $\Hom ( X^* (P), \bZ )$ is typically a larger group than $\Hom ( X^* (G), \bZ )$. This will be illustrated in the following example, which we also consider in order to compare the notions of topological type and degree.
\end{remark}

\begin{example} \label{ParabolicInGLr}
Take $G = \GL_r$. In Example \ref{example:reductions} (a), we considered the parabolic subgroup $P$ of matrices of the form
\[
\begin{pmatrix}
A_1 & \star & \cdots & \star \\
0 & A_2 & \cdots & \star \\
\vdots & & \ddots & \vdots \\
0 & 0 & \cdots & A_t
\end{pmatrix}
\]
where $r_1 + \cdots + r_t = r$ and $A_i \in \GL_{r_i}$. Then
\[
X^* (P) \ = \ \left\{ {\det}_1^{k_1} \cdots {\det}_t^{k_t} : k_1 , \ldots , k_t \in \bZ \right\}
\]
where ${\det}_i$ is the character taking the above matrix to $\det A_i$, and
\[
\Image \left( X^* (G) \ \to \ X^* (P) \right) \ = \ \left\{ \left( {\det}_1 \cdots {\det}_t \right)^k : k \in \bZ \right\} .
\]
Let now $E_P \to X$ be a principal $P$-bundle and $E := E_P (G)$ the associated $G$-bundle. As discussed in Example \ref{example:reductions} (a), the associated vector bundle $V := E ( \C^r )$ has a filtration $0 = V_0 \subset V_1 \subset \cdots \subset V_{t-1} \subset V_t = V$. For $1 \le i \le t$, write $W_i := V_i / V_{i-1}$ for the summands of the associated graded bundle. Let $\phi \colon P \hookrightarrow G$ be the inclusion. Now
\begin{multline} \label{dPdG}
\left\langle \phi_* d_P (E_P) ,  {\det}^k \right\rangle \ = \ \left\langle d_P (E_P) , {\det}_1^k \cdots {\det}_t^k \right\rangle \ = \ k \cdot \deg W_1 + \cdots + k \cdot \deg W_t \ = \\
k \cdot \left( \deg W_1 + \cdots + \deg W_t \right) \ = \ k \cdot \deg V \ = \ \left\langle d_G (E) , {\det}^k \right\rangle .
\end{multline}

Let us relate the above discussion with the topological types of $E_P$ and $E$. Recall the exact sequence $1 \to \Ru P \to P \to L \to 1$ as in (\ref{UPLsequence}). The unipotent group $\Ru P$ is isomorphic as a variety to an affine space, and the Levi $L$ is the product $\prod_i \GL_{r_i}$. As the map $P \to L$ is a fibration with connected and simply connected fibres, Serre's long exact sequence of homotopy groups shows that the induced map $\pi_1 (P) \to \pi_1 (L)$ is an isomorphism. Thus the topological type of a $P$-bundle is determined by the degrees of the vector bundles appearing along the diagonal. Moreover, the map $\pi_1 (P) \to \pi_1 (\GL_r)$ induced by the inclusion $P \hookrightarrow \GL_r$ corresponds to adding the degrees of the $\GL_{r_i}$-bundles occurring along the diagonal, in analogy with the computation (\ref{dPdG}).

In this case we have $\pi_1 (P) \cong \Hom (X^* (P), \bZ)$ and $\pi_1 (G) \cong \Hom (X^* (G), \bZ)$. However, this need not hold for arbitrary $P$ and $G$; for example, the groups $\pi_1 ( P )$ and $\pi_1 (G)$ may have torsion, while $\Hom ( \Lambda, \bZ )$ is torsion free for any group $\Lambda$ since $\bZ$ is (see Remark \ref{TopTypeFiner} for an example). The precise relation between topological type and degree will be a subject of further investigation.
\end{example}

\section{Segre invariant for principal bundles} \label{FirstProps}

Henceforth, we will assume that $G$ is a connected reductive group. Let $E \to X$ be a principal $G$-bundle. For a parabolic subgroup $P \subset G$, not necessarily maximal, write $\pi \colon E/P \to X$ for the corresponding homogeneous space bundle (see \S \ref{sec:principalGbundles}). We now discuss further the Segre function set out in Definition \ref{sPDefn}. The term ``invariant'' will be justified in {\S} \ref{invariance}.

\subsection{Characterisations of \texorpdfstring{$s_P$}{s\_P}} \label{characterisations}

Here we give some equivalent definitions of the Segre function defined in the introduction
\[
s_P (E) \ = \ \min \{ \deg \sigma^* T_\pi : \sigma \colon X \to E/P \hbox{ a reduction of structure group to } P \} ;
\]
essentially, we are just making explicit some statements in \cite[p.\ 322--323]{HN}. Firstly, we show that the above is equivalent to the definition considered in \cite{HN} and formalised in \cite[p.\ 193]{CH1}, which we will use in the proof of semicontinuity of $s_P$ (Theorem \ref{semicont}). Abusing notation, given a map of varieties $f \colon Y \to Z$, we write $df$ both for the induced map $T_Y \to T_Z$ and that $T_Y \to f^* T_Z$, when these exist. The proof of the following lemma is left to the reader:

\begin{lemma} \label{lemma:normal}
Let $\pi \colon F \to Y$ be a smooth morphism of smooth varieties. Set
\[
T_\pi \ := \ \Ker \left( T_F \to \pi^* T_Y \right) \ \cong \ \Ker \left( d\pi \colon T_F \to T_Y \right) .
\]
Suppose that $\sigma \colon Y \to F$ is such that $\pi \circ \sigma = \Iden_Y$. Then $\sigma^* T_F$ splits canonically as $\sigma^* T_\pi \oplus T_Y$, and the normal bundle $N_\sigma := N_{\sigma(Y) / F}$ is isomorphic to $\sigma^* T_\pi$.
\end{lemma}

By Lemma \ref{lemma:normal}, we have
\begin{equation}\label{sPviaNumType}
s_P (E) \ = \ \min \left\{ \deg N_\sigma : \sigma \colon X \to E/P \hbox{ a reduction of structure group to } P \right\} .
\end{equation}

A variant of the above will also be used in Theorem \ref{semicont}. Firstly, we recall a definition. As the adjoint action of $P$ on $\fg$ preserves the subspace $\fp$, there is an exact sequence of $P$-modules 
\begin{equation} \label{pModSeq}
0 \ \to \ \fp \ \to \ \fg \ \to \ \fg/\fp \ \to \ 0 \; .
\end{equation}

\begin{definition} \label{IsotropyRepDefn}
The map $\iota \colon P \to \GL ( \fg / \fp )$ obtained from the above sequence is called the \emph{isotropy representation} of $P$.
\end{definition}

Now let $E$ and $\pi \colon E/P \to X$ be as above. As before, write $E_{P, \sigma} \to X$ for the principal $P$-bundle coming from a reduction $\sigma$. As described in \cite[{\S} 2]{Rth1}, the sequence (\ref{pModSeq}) induces an exact sequence of associated vector bundles
\begin{equation} \label{AdjSeq}
0 \ \to \ E_{P, \sigma} ( \fp ) \ \to \ E_{P, \sigma} ( \fg ) \ \to \ E_{P, \sigma} ( \fg / \fp ) \ \to \ 0 , 
\end{equation}
where $E_{P, \sigma} ( \fg / \fp )$ is identified with $\sigma^* T_\pi$. Thus
\begin{equation}\label{sPadjoint}
s_P (E) \ = \ \min \left\{ \deg E_{P, \sigma} ( \fg / \fp ) : \sigma \colon X \to E/P \hbox{ a reduction of structure group} \right\} .
\end{equation}
Observe that the vector bundle $E_{P, \sigma} ( \fg )$ coincides with the adjoint bundle $E ( \fg )$, since they have the same transition functions. 

Lastly, we reformulate (\ref{sPadjoint}) in terms of characters, to maintain contact with \cite{Hol02} and for use in $\S$ \ref{sec:Borel}. The determinant $\det \iota$ of the isotropy representation is an element of $X^* (P)$. Given a $G$-bundle $E$ with $P$-reduction $\sigma$, as in Definition \ref{def_num_type} we have the numerical type $[ \sigma ] = d_P (E_{P, \sigma}) \in \Hom ( X^* (P), \bZ )$. Evaluating this on $\det \iota$, we obtain
\[
\deg \left( E_P ( \det \iota , \C ) \right) \ = \ \deg \left( \det E_P ( \fg / \fp ) \right) \ = \ \deg \left( E_P ( \fg / \fp ) \right) .
\]
Thus we have another interpretation of $s_P (E)$ given by
\begin{equation}\label{sPpairing}
s_P (E) \ = \ \min \left\{ \langle [ \sigma ] , \det \iota \rangle : \sigma \colon X \to E/P \hbox{ a reduction of structure group} \right\} ,
\end{equation}
where $\langle \; ,\, \rangle$ is the natural pairing $X^* (P) \times \Hom ( X^* (P), \bZ ) \to \bZ$.

\subsection{Semicontinuity of the Segre invariant and stratification of the moduli of principal bundles}

We prove here Theorem \ref{thm:mainA} (i), that $s_P$ is a semicontinuous function in families of $G$-bundles. We use several ideas from \cite{HN}.

\begin{theorem} \label{semicont}
Let $G$ be a connected reductive algebraic group. Let $\cE \to B \times X$ be a family of principal $G$-bundles over a curve $Y$, where $B$ is a scheme of finite type. For each parabolic subgroup $P$ of $G$ and for each integer $s$, the set
\begin{equation} \label{Bs}
\left\{ b \in B : s_P ( \cE_b ) \le s \right\}
\end{equation}
is closed in $B$. In particular, $s_P$ is lower semicontinuous in families of $G$-bundles. 
\end{theorem}

\begin{proof}
The associated bundle $\pi \colon \cE/P \to B \times X$ is a family of $G/P$-fibrations. As in \cite[Lemma 2.1]{HN}, we consider the family of associated vector bundles $\cE ( \fg )$ over $B \times X$. If $\sigma \colon \{ b \} \times X \to \cE_b / P$ is a reduction for some $b \in B$, then by (\ref{AdjSeq}) there is a surjective map
\[
\cE_b ( \fg ) \ \to \ \left( \cE_b \right)_P ( \fg/\fp ) \ \cong \ \sigma^* T_{\pi_b}
\]
of bundles on $X$. Now $\sigma^* T_{\pi_b}$ has rank $\dim G - \dim P$. Since $B$ is bounded, there is a constant $k$ such that for all $b \in B$, all quotients $W$ of $\cE_b ( \fg )$ of this rank satisfy $\deg W \ge k$. In particular there is a uniform lower bound for $s_P ( \cE_b )$ on all of $B$. Therefore, there are at most finitely many $s' \le s$ such that $s_P ( \cE_b ) = s'$ for some $b \in B$.

For any such $s'$, suppose for some $b \in B$ that $\sigma \colon \{ b \}\times X \to \cE_b/P$ is a reduction of structure group with $\deg \sigma^* T_{\pi_b} = s'$. Denote by $q$ the Hilbert polynomial of $\sigma (X)$ in $\cE_b / P$ with respect to some polarisation of $\cE_b / P$. Consider the relative Hilbert scheme
\[
\cH ilb^q \left( ( \cE / P ) / B \right) \ \to \ B
\]
whose fibre at $b \in B$ is $\Hilb^q ( \cE_b / P )$. Write $\cH$ for the component containing the Hilbert point of $\sigma (X)$, which is also a scheme over $B$.

Let $\cX \to \cH \times \cE/P$ be the universal curve. We claim that for any $h \in \cH$, the scheme $\cX_h$ is a curve
\begin{equation} \label{reducible}
D \cup D_1 \cup \cdots \cup D_k
\end{equation}
in $\cE_{b'} / P$ for some $b' \in B$, where $D = \tau(X)$ for some section $\tau$ of $\cE_{b'} / P$ satisfying $\deg (\tau^* T_{\pi_{b'}} ) \le s'$, and $D_1 , \ldots , D_k$ are curves lying in fibres of $\cE_{b'} / P \to X$. To see this, we use the same argument as in the proof of \cite[Lemma 2.2]{HN}. We have a diagram
\[ \xymatrix{
 \cX \ar[r]^-\rho & \cH \times \cE / P \ar[r]^-{p_1} \ar[d]^{p_2} & \cH \ar[dr] & \\
 & \cE / P \ar[r]^\pi & B \times X \ar[r] \ar[d]_{p_X} & B \\
 & & X & .
} \]
Now $\rho^* p_2^* T_\pi$ is locally free and hence flat over $\cX$ by \cite[III.9.2 (e)]{Har}. By the general theory of Hilbert schemes, $\cX \to \cH$ is a flat family of curves. Hence $\rho^* p_2^* T_\pi$ is also flat over $\cH$ by \cite[III.9.2 (c)]{Har}. Therefore, $\chi ( \cX_h , (\rho^* p_2^* T_\pi)|_{\cX_h} )$ is independent of $h \in \cH$. Similarly, $\chi ( \cX_h , \cO_{\cX_h} ) = \chi ( X, \Ox )$ for all $h \in \cH$, and $\chi ( \cX_h , \rho^* p_2^* \pi^* p_X^* L ) = \chi ( X, L )$ for any line bundle $L$ of degree $1$ over $X$. By \cite[Lemma 2.2]{HN}, for any $h \in \cH$, the curve $\cX_h$ has the form stated in (\ref{reducible}).

Now as noted in the proof of \cite[Lemma 2.1]{Rth1}, the determinant of the isotropy representation (Definition \ref{IsotropyRepDefn}) is a dominant character. Thus if $\cX_h$ has components of the form $D_i$ in (\ref{reducible}) the line bundle $\det ( T_{\pi_{b'}}|_{D_i} )$ is ample for each $i$ by the Kempf vanishing theorem \cite{Ha80}, and thus of positive degree. Therefore,
\[
s_P ( \cE_{b'} ) \ \le \ \deg ( T_{\pi_{b'}}|_D ) \ \le \ \deg ( T_{\pi_{b'}}|_{\cX_h} ) \ \le \ s' \ \le \ s
\]
for all $h \in \cH$.

It follows that the locus (\ref{Bs}) is precisely the union of the images of all components of the form $\cH$ as $s'$ ranges over the finitely many values discussed at the beginning. As each of the Hilbert schemes in question is proper over $B$, the image of each $\cH$ is a closed subset. Thus (\ref{Bs}) is a union of finitely many closed subsets of $B$, so is closed. This completes the proof. \end{proof}

As before, let $\MGd$ be the moduli space of stable $G$-bundles of topological type $\delta$ and define the subsets
\begin{equation}\label{eq-strata}
\Ms \ := \ \left\{ E \in \MGd : s_P (E) \le s \right\}
\end{equation}
of those stable $G$-bundles whose Segre invariant relative to $P$ is less than or equal to $s$. Also, one can consider the loci
\begin{equation}\label{eq-strata-circ}
\Mso \ := \ \{ E \in \MGd : s_P (E) = s \}
\end{equation}
which, by Proposition \ref{semicont}, are locally closed. We will omit $\delta$ in the notation of (\ref{eq-strata}) and (\ref{eq-strata-circ}) when $\pi_1(G)$ is trivial. The following is immediate from Proposition \ref{semicont} and the functorial properties of the moduli of principal bundles.

\begin{corollary}\label{cor-stratification}
For each $\delta \in \pi_1 (G)$, the loci $\Mso$ define a stratification of the moduli space $\MGd$ into disjoint locally closed subsets.
\end{corollary}

\subsection{Invariance properties of \texorpdfstring{$s_P$}{s\_P}} \label{invariance}

Let $V \to X$ be a vector bundle of rank $r$ and let $L \to X$ be any line bundle. It is easy to check that for $1 \le n \le r-1$ we have
\[
s_n (V \otimes L) \ = \ s_n (V)
\]
for all line bundles $L \to X$. 
 In particular, $s_n (V)$ depends only on the projective bundle $\PP V$; and this is an associated fibration of $E (\PGL_r)$, the principal $\PGL_r$-bundle obtained by extension of structure group via the adjoint map $\GL_r \to \PGL_r = \GL_r / Z(\GL_r )$. We will now generalise this invariance property of $s_n$, comparing the functions $s_P$ and $s_{\phi(P)}$ for a surjective homomorphism $\phi$. Firstly, we require a technical lemma. For an algebraic group $K$, we denote by $K_0$ the connected component containing the identity element.

\begin{lemma} \label{parprops}
Let $\phi \colon G \to H$ be a surjective homomorphism of connected algebraic groups.
\begin{enumerate}
\item[(a)] Let $Q$ be a parabolic subgroup of $H$. Then there is an isomorphism of $H$-spaces $G/\phi^{-1}(Q) \xrightarrow{\sim} H/Q$. In particular, $\phi^{-1}(Q)$ is parabolic in $G$.
\item[(b)] Let $P \subset G$ be a parabolic subgroup containing $\Ker (\phi)_0$. Then there is an isomorphism of $H$-spaces $G/P \xrightarrow{\sim} H/\phi(P)$.
\end{enumerate}
\end{lemma}

\begin{proof}
(a) Write $Q' := \phi^{-1}(Q)$. Using surjectivity of $\phi$ and normality of $\Ker (\phi)$, one checks that $H$ acts on $G/Q'$ by
\[ ( h , (g \cdot Q') ) \ \mapsto \ \tilde{h} g \cdot Q' \]
for any lift $\tilde{h}$ of $h$; and that there is a well defined map $G/Q' \to H/Q$ given by $g \cdot Q' \mapsto \phi(g) \cdot Q$. Unwinding definitions, this can be checked to define the desired isomorphism of $H$-spaces.

(b) By \cite[Corollary 6.2.8]{Spr}, the image $\phi(P)$ is a parabolic subgroup of $H$. Therefore, part (a) shows that $P' := \phi^{-1}(\phi(P))$ is a parabolic subgroup. Thus (b) will follow if we can show that $P' = P$.

Since $\Ker (\phi)_0 \subseteq P$, we have $\dim P' = \dim P$, so the coset space $P'/P$ is a finite set. Since $G$ is connected and $P'$ is a parabolic subgroup, $P'$ is connected by \cite[Corollary 6.4.10]{Spr}. Therefore, $P'/P$ is a single point and $P' = P$, as desired.
\end{proof}

\begin{remark} \label{Kerphicontained}
The proof of (b) shows that if $\Ker ( \phi )_0 \subseteq P$, then in fact all components of $\Ker ( \phi )$ belong to $P$, since $\Ker (\phi) \subseteq \phi^{-1} (J)$ for any subgroup $J \subseteq H$. Thus in this case we also have $\phi^{-1} (\phi (P)) = P$.
\end{remark}

\begin{example}
Suppose that $G = \GL_r$ and $H = \PGL_r$ and $\phi \colon \GL_r \to \PGL_r$ is the adjoint map. Let $P \subset \GL_r$ be the maximal parabolic subgroup stabilising a subspace $\C^n \subset \C^r$. Then the isomorphism $G/P \isom H/\phi(P)$ in Lemma \ref{parprops} (b) is the natural identification between the Grassmannian variety parameterising $n$-planes in $\C^r$ and that parameterising projective subspaces of dimension $n - 1$ in $\PP^{r-1}$.
\end{example}

We can now state the main invariance property of $s_P$.

\begin{proposition} \label{sPinv}
Let $\phi \colon G \to H$ be a surjective homomorphism of connected algebraic groups, and let $P \subset G$ be a parabolic subgroup containing $\Ker(\phi)_0$.
\begin{enumerate}
\item[(a)] Let $E \to X$ be a principal $G$-bundle. Then $s_P (E) = s_{\phi(P)}( E(H) )$.
\item[(b)] Let $\cE \to B \times X$ be a family of $G$-bundles. Then
\[
\{ b \in B : s_P ( \cE_b ) = s \} \ = \ \{ b \in B : s_{\phi(P)} ( \cE_b (H) ) = s \} .
\]
\end{enumerate}
\end{proposition}

\begin{proof}
(a) By Lemma \ref{parprops} (b) and {\S} \ref{AssocBundles}, there is an isomorphism of fibrations over $X$ in $H$-homogeneous spaces
\[
E/P \ \cong \ E (G/P) \ \isom \ E (H) \left( H/\phi(P) \right) .
\]
Thus there is a bijection (depending on $\phi$) between sections of $E/P$ and sections of $E (H) \left( H/\phi(P) \right)$ preserving the isomorphism class of the restricted vertical tangent bundles. Therefore,
\begin{multline*}
s_P (E) \ = \ \min \left\{ \deg \sigma^* T_{(E/P)/X} : \sigma \hbox{ a section of } E/P \to X \right\} \ = \\
 \min \left\{ \deg \tau^* T_{\left( E(H)/\phi(P) \right) /X} : \tau \hbox{ a section of } E(H) \left( H/\phi(P) \right) \to X \right\} \ = \ s_{\phi(P)} \left( E(H) \right) ,
\end{multline*}
as desired. Part (b) is immediate from (a).
\end{proof}

As a special case, we note that $s_P (E)$ is an invariant of the conjugacy class of $P$, proving Theorem \ref{thm:mainA} (ii). This slightly generalises \cite[Remark 2.1]{Rth1}.

\begin{corollary} \label{ConjInv}
Let $E \to X$ be a principal $G$-bundle. For any $h \in G$ there is a $G$-bundle $E'$ and an equivariant isomorphism $\iota \colon E \isom E'$ such that $s_P (E) = s_{hPh^{-1}} (E')$.
\end{corollary}

\begin{proof}
Given an element $h \in G$, extension of structure group by the inner automorphism $g \mapsto h g h^{-1}$ gives a $G$-bundle $E' := E(hGh^{-1})$ isomorphic to $E$. By Proposition \ref{sPinv} (a), we have $s_P (E) \ = \ s_{hPh^{-1}} ( E (hGh^{-1}) ) = s_{hPh^{-1}}(E')$.
\end{proof}

\section{Segre stratifications of moduli spaces of principal \texorpdfstring{$G$}{G}-bundles}\label{sec:stratificacions}

We denote by $\MGd$ the moduli space of stable principal $G$-bundles of topological type $\delta$, and by $M_X(G)$ when $\pi_1(G)$ is trivial and there is no choice for $\delta$ other than the trivial topological type. From Corollary \ref{cor-stratification}, the moduli $\MGd$ stratifies as a disjoint union of the locally closed subsets 
\[\Mso \ := \ \{ E \in \MGd : s_P (E) = s \}\; .\]
We are interested in answering questions on nonemptiness, irreducibility and dimension of the strata $\Mso$ together with, for a given $s$, which strata $M_X (G, \delta; P, s')$ lie in the closure $\overline{\Mso}$. For classical $G$ with maximal parabolics $P$, we can refer to the literature:

Set $G = \GL_r$ and recall from (\ref{Segrestratum-vectorbundles}) the Segre stratum $U_X (r, d; n, s)$ for vector bundles. The following result can be found in \cite{RT,BL}.

\begin{theorem}\cite[Theorem 0.1 and Theorem 0.2]{RT}; \cite[Theorem 4.2]{BL}.\label{GLrCase}
Suppose that $0 < s \le n(r-n)(g-1) + (n-1)$ and $s \equiv nd \;(\!\!\mod r)$. Then the Segre stratum $U_X (r, d; n, s)$ is irreducible of dimension $(r^2 - n(r-n))(g-1) + s + 1$. Moreover, $U_X (r, d; n, s')$ belongs to the closure of $U_X (r, d; n, s)$ if $s' \equiv n d \;(\!\!\mod r)$ and $0 < s' \le s$.
\end{theorem}

Let $P \subset \GL_r$ be the stabiliser of a subspace of dimension $n < r$. Identifying $\delta \in \pi_1 (\GL_r) = \bZ$ with the degree of the associated vector bundles (Definition \ref{defn:degree}), it follows from (\ref{sPsndef}) that the definition of the Segre function $s_P$ for principal $\GL_n$-bundles coincides with the Segre function $s_n$ for the associated vector bundles, therefore $M_X ( \GL_r, \delta; P, s )^\circ$ is nonempty only if $s \equiv n \delta \;(\!\!\mod r)$; and that $M_X ( \GL_r, \delta; P, s' )$ lies in the closure of $M_X ( \GL_r, \delta; P, s )^\circ$ if $s' \equiv n \delta \;(\!\!\mod r)$ and $0 < s' \le s$.

For a different example, set $G = \SO_{2n}$. We have $\pi_1 ( \SO_{2n} ) = \bZ_2$. Thus $M_X ( \SO_{2n} )$ has two components, distinguished by the Stiefel--Whitney class $w_2 \in H^2 (X, \bZ_2 ) = \bZ_2$. Let $Q$ be the maximal parabolic stabilising an isotropic subspace $\Lambda \subset \C^{2n}$ of dimension $n$. By Lie theory, $\fg / \fq \cong \wedge^2\Lambda^*$, the tangent space of the Lagrangian Grassmannian $\OG (n, \C^{2n} )$ at $\Lambda$; and for a reduction $\sigma \colon X \to E/Q$, we have $E_{Q, \sigma} ( \fg/\fq ) \cong \wedge^2 F^*$, where $F \subset E ( \C^{2n} )$ is the isotropic subbundle corresponding to the reduction. Thus
\begin{equation}\label{sP-symp-def}
s_Q (E) \ = \ \min\{ -(n-1) \cdot \deg F : F \subset E ( \C^{2n} ) \hbox{ an isotropic subbundle of rank } n \} .
\end{equation}
which coincides with (\ref{sn-symp-def}) up to rescaling.

Then \cite{CH2} gives the following analogue of Theorem \ref{GLrCase}, which we reformulate in terms of principal bundles:

\begin{theorem}\cite[Theorem 1.3]{CH2} \label{OrthStrat}
Suppose that
\begin{equation} \label{conditions-Spin}
0 \ < \ s \ \le \ \frac{1}{2}(n-1)n(g-1) + 2 \quad \hbox{and} \quad s \equiv 0 \;(\!\!\mod n-1) ,
\end{equation}
Then the image of $M_X (\SO_{2n}, \delta ; Q, s)^\circ$ in $M_X (\SL_{2n})$ is irreducible and of dimension \linebreak $\frac{1}{2}n(3n - 1)(g - 1) + s$.
\end{theorem}

Analogous statements for the symplectic and odd orthogonal cases are proven in \cite[Theorem 1.1]{CH2} and \cite[Theorem 5.1]{CH3} respectively.

In the remainder of this section, we will consider certain homomorphisms of algebraic groups $\phi \colon G \to H$ which interact well with stability of $G$- and $H$-bundles. This will allow us to transfer information on the stratification in the moduli space of principal $G$-bundles to the moduli of $H$-bundles, extending the literature to a number of other groups. 

\subsection{Maps between moduli spaces}

Let $\phi \colon G \rightarrow H$ be a surjective morphism of connected reductive algebraic groups. Here we will generalise some statements in \cite[{\S} 7]{Rth1}, to relate the stability of a principal $G$-bundle $E$ and that of $E(H)$. Firstly, we record an essential technical fact.

\begin{lemma} \label{PcontainsZ}
Let $G$ be a connected linear algebraic group, and let $P \subset G$ be a parabolic subgroup. Then $Z(G) = Z(P)$. (In particular, $Z(G) \subseteq P$.)
\end{lemma}

\begin{proof}
By Springer \cite[Corollary 6.4.10]{Spr}, we have $N_G (P) = P$. As $Z(G) \subseteq N_G (H)$ for all subgroups $H \subseteq G$, in particular $Z(G) \subseteq P$; whence $Z(G) \subseteq Z(P)$. For the converse, we observe that the argument of \cite[Corollary 6.2.9]{Spr} applies to an arbitrary parabolic: If $y \in Z(P)$ then the commutator map $g \mapsto y g y^{-1} g^{-1}$ defines a map $G \to G$ which factorises via $G/P$; for if $p \in P$, then
\[
y ( g p ) y^{-1} (g p)^{-1} \ = \ y g (p y^{-1} p^{-1} ) g^{-1} \ = \ y g (y^{-1} p p^{-1} ) g^{-1} \ = \ y g y^{-1} g^{-1} .
\]
But since $G/P$ is projective and $G$ is affine, the image of this map is constant, and thus equals the point $y \cdot 1_G \cdot y^{-1} \cdot 1_G^{-1} = 1_G$. Thus in fact $y \in Z(G)$.
\end{proof}

The main result of the section proves Theorem \ref{thm:mainB} and describes a special situation where all Segre functions $s_P$ are invariant under a surjective homomorphism, and so we obtain a morphism between certain components of the moduli spaces and a correspondence between the stratifications. In particular, this generalises \cite[Proposition 7.1]{Rth1} and \cite[Corollary 3.18]{Rth2}. Recall the locally closed strata $\Mso$ defined in (\ref{eq-strata-circ}). For each $\delta \in \pi_1 (G)$, define
\begin{equation}\label{sigma-strata}
\Sigma (G, \delta; P) \ := \ \{ s \in \bN : \Ms^\circ \hbox{ is nonempty} \} .
\end{equation}

\begin{theorem} \label{RelateStrata}
Let $\phi \colon G \to H$ be a surjective homomorphism of connected reductive algebraic groups satisfying $\Ker (\phi)_0 \subseteq Z(G)$. Let $E \to X$ be a principal $G$-bundle and let $E(H)$ be the associated principal $H$-bundle.
\begin{enumerate}
\renewcommand{\labelenumi}{(\alph{enumi})}
\item For \emph{all} parabolics $P \subset G$, we have $s_P (E) = s_{\phi(P)} (E(H))$.
\item $E$ is (semi)stable if and only if $E(H)$ is (semi)stable.
\item For each topological type $\delta \in \pi_1 (G)$, the map $\phi_\delta \colon M_X (G, \delta ) \ \to \ M_X (H, \phi_* \delta )$ induced by extension of structure group is a morphism.
\item We have $\phi_\delta^{-1} M_X (H, \phi_* \delta ; \phi (P), s)^\circ = M_X (G, \delta ; P, s)^\circ$.
\item If $\phi_\delta$ is surjective, then $\Sigma (G, \delta; P) = \Sigma (H, \phi_* \delta ; \phi (P) )$.
\end{enumerate}
\end{theorem}

\begin{proof}
(a) By Lemma \ref{PcontainsZ}, for all parabolics $P \subset G$ we have $\Ker (\phi)_0 \subseteq P$. Thus $s_P (E) = s_{\phi(P)} (E(H))$ for all $P$ by Proposition \ref{sPinv} (a).

(b) Observe that, for each parabolic $P\subset G$, the image $\phi(P)$ is parabolic in $H$ by \cite[Corollary 6.2.8]{Spr}; and, for a parabolic subgroup $Q \subset H$, the inverse image $\phi^{-1} (Q)\subset G$ is parabolic in $P$ by Lemma \ref{parprops} (a). Also, recall that $E$ is stable if $s_P(E) > 0$ for every maximal parabolic $P\subset G$ and strictly semistable if $s_P(E)\geq 0$ for every maximal parabolic $P$ with equality for at least one $P$. 
Then the statement follows from Proposition \ref{sPinv} (a) and Remark \ref{Kerphicontained}, since
\[
s_Q (E(H)) \ = \ s_{\phi (\phi^{-1} (Q))} (E(H)) \ = \ s_{\phi^{-1} (Q)} (E) \quad \hbox{and} \quad s_P (E) \ = \ s_{\phi(P)}(E(H)) .
\]

(c) This follows directly from the ``only if'' of part (b).

(d) Suppose $F \in M_X (H, \phi_* \delta )$ is a principal $H$-bundle with $s_{\phi(P)} (F) = s$, and which is of the form $E(H)$ for some $E \in M_X (G, \delta)$; that is, lying in the image of $\phi_\delta$. By (a), we have $s_P (E) = s$. Conversely, if $s_P (E) = s$ for a $G$-bundle $E$ of topological type $\delta$, then $s_{\phi(P)} ( E(H) ) = s$ by the same argument.

(e) By (d), we have an inclusion $\Sigma (G, \delta; P) \subseteq \Sigma (H, \phi_* \delta ; \phi (P) )$ in any case. When $\phi_\delta$ is surjective, by part (a) we see that $\Mso$ is nonempty whenever $M_X (H, \phi_* \delta ; \phi (P) , s )^\circ$ is nonempty. Thus the inclusion is a bijection.
\end{proof}

\subsection{Examples}\label{sec:examples}

Here we give examples showing how the above invariance results allow us to transfer information between the Segre stratifications of $M_X (G, \delta)$ and $M_X (H, \phi_*\delta)$ when $\phi \colon G \to H$ is a quotient by a subgroup of $Z(G)$. The situations presented here essentially cover all essential situations of algebraic groups associated with a simple Lie algebra of Dynkin type A or C in $\S$\ref{GLrPGLr}, and types B or D in $\S$\ref{SpinSO}, ranging from the simply connected to the adjoint case. 

\subsubsection{\texorpdfstring{$\GL_r$}{GL\_r} and \texorpdfstring{$\PGL_r$}{PGL\_r}} \label{GLrPGLr}

Let $\phi \colon \GL_r \to \PGL_r $ be the adjoint map. In this case, \linebreak $\phi_* \colon \pi_1 (\GL_r) \to \pi_1 (\PGL_r)$ can be identified with  the reduction mod $r$ map $\bZ \to \bZ_r$. We write $\overline{\delta} := \delta \;(\!\!\mod r)$.

As $X$ is a curve, Tsen's theorem implies that every $\PGL_r$-bundle over $X$ lifts to a $\GL_r$-bundle. Moreover, by Theorem \ref{RelateStrata} (b), such liftings preserve stability. Therefore, for each $\delta \in \bZ$ the morphism $\phi_\delta \colon M_X ( \GL_r , \delta ) \to M_X ( \PGL_r, \overline{\delta} )$ is surjective. Moreover, it follows from \cite[Proposition 5.5.2]{Gro} that a general fibre $\phi_\delta^{-1} ( E )$ is a torsor over $\Pic^0 (X)$, and in particular of dimension $g$. (This also follows from the fact that two vector bundles $V$ and $V'$ give rise to isomorphic projective bundles if and only if $V' \cong V \otimes L$ for some line bundle $L$).

Now let $P_n \subseteq \GL_r$ be the maximal parabolic preserving a fixed $\C^n \subset \C^r$. Then $\phi (P_n)$ is a parabolic subgroup of $\PGL_r$ by Lemma \ref{parprops}. By the last paragraph and by Theorem \ref{GLrCase}, we deduce:

\begin{quote}
Let $\epsilon \in \{ 0 , \ldots , r-1 \}$ be such that $n(r-n)(g-1) + \varepsilon \equiv n \delta \;(\!\!\mod r)$. Then if $0 < s \le n(r-n)(g-1) + \epsilon$ and $s \equiv n \delta \;(\!\!\mod r)$, the locus $M_X (\PGL_r, \overline{\delta} ; \phi (P_n ), s )^\circ$ is
irreducible and of dimension
\[
(r^2 - n(r-n))(g-1) + s + 1 - g \ = \ (r^2 - n(r-n) - 1)(g - 1) + s .
\]
\end{quote}

\begin{remark} \label{TopTypeFiner}
As mentioned at the end of Example \ref{ParabolicInGLr}, topological type is in general a finer invariant than numerical type. Here, for example, $\pi_1 ( \PGL_r ) = \bZ_r$, while $\Hom ( X^* (\PGL_r), \bZ )$ is trivial since $X^* ( \PGL_r )$ is.
\end{remark}

\begin{remark}
Along the lines of the previous example: Suppose that $m$ divides $r$. Let $K:=\bZ_m\subset Z(\SL_r)=\bZ_r$ and consider the surjective adjoint map of groups $\phi \colon \SL_r/K\rightarrow \PGL_r$. At the level of fundamental groups the map $\phi_* \colon \pi_1(\SL_r/K)=\bZ_{r/m}\rightarrow \pi_1(\PGL_r)=\bZ_r$ can be identified with $\delta\mapsto m\delta$, the multiplication by $m$ (the order of $K$).

By Theorem \ref{RelateStrata}, for each $\delta\in\pi_1(\SL_r/K)=\bZ_{r/m}$ we obtain a morphism
\[
\phi_\delta \colon M_X (\SL_r/K, \delta ) \ \to \ M_X (\PGL_r, m\delta )
\]
such that $\phi_\delta^{-1} M_X (\PGL_r, m\delta ; \phi (P), s)^\circ = M_X (\SL_r/K, \delta ; P, s)^\circ$. In particular, for trivial $K$ we obtain a morphism $\phi_1 \colon M_X (\SL_r) \ \to \ M_X (\PGL_r, 0)$ from the unique component of the moduli space of principal $\SL_r$-bundles to the component of the moduli of $\PGL_r$-bundles of trivial topological type, which preserves the Segre stratification. 
\end{remark}

\begin{remark}
In a similar vein, the surjective adjoint homomorphism $\phi: \Sp_{2n}\to \PSp_{2n}$, where $\PSp_{2n}=\Sp_{2n}/Z(\Sp_{2n})=\Sp_{2n}/\bZ_2$, provides us with a surjective morphism \linebreak $\phi_0 \colon M_X (\Sp_{2n}) \ \to \ M_X (\PSp_{2n}, 0)$ from the unique component of the moduli space of principal $\Sp_{2n}$-bundles to the component of the moduli of $\PSp_{2n}$-bundles of trivial topological type $\delta = 0$, preserving the Segre stratification as $\phi_0^{-1} M_X (\PSp_{2n}, 0 ; \phi (P), s)^\circ = M_X (\Sp_{2n}; P, s)^\circ$.
\end{remark}

\subsubsection{\texorpdfstring{$\Spin_{2n}$}{Spin\_2n} and \texorpdfstring{$\SO_{2n}$}{SO\_2n}}\label{SpinSO}

For $r \geq 3$, consider the universal covering sequence \linebreak $1 \to \bZ_2\to \Spin_r \xrightarrow{\phi} \SO_r \to 1$. As $\Spin_r$ is simply connected, $M_X (\Spin_r)$ has just one component. The obstruction to lifting an $\SO_r$-bundle $E$ to a $\Spin_r$-bundle is the second Stiefel--Whitney class $w_2 (E) \in H^2 (X, \bZ_2) = \bZ_2$. As $\bZ_2 \subseteq Z ( \Spin_r )$ for all $r \ge 3$, again by Theorem \ref{RelateStrata} (b) liftings preserve stability, so the map $\phi_0: M_X (\Spin_r) \to M_X (\SO_r, 0)$ is surjective. Furthermore, by \cite[pp.\ 79--80]{Gro}, each fibre of $\phi_0$ admits a transitive action of $H^1 (X, \bZ_2)$, which is identified with the group $\Pic^0 (X)[2]$ of $2$-torsion line bundles over $X$. Thus $\phi_0$ is quasi-finite.

Suppose now that $r = 2n$ is even. Let $P$ be the maximal parabolic subgroup stabilising an isotropic subspace $\Lambda \subset \C^{2n}$ of dimension $n$. (Note that there are two conjugacy classes of such $P$; we return to this point below.) Let us use the results of \cite{CH2} to give the dimension of the strata $M_X (\Spin_{2n} ; P, s)$. Consider the map $\psi: M_X (\SO_{2n} , 0) \ \to \ M_X ( \SL_{2n} )$
given by extension of structure group. By \cite[{\S} 2]{Ser}, this is a ramified double cover. It follows that the composed map $\psi\circ \phi_0: M_X ( \Spin_r ) \to M_X (\SL_{2n})$ is quasi-finite\footnote{Note that \cite{CH2} uses the language of vector bundles; here we are applying the natural identification between $M_X (\SL_{2n})$ and the moduli space of stable vector bundles of rank $2n$ and trivial determinant.}.

Therefore, by Theorem \ref{RelateStrata} (d) and Theorem \ref{OrthStrat}, for $s$ satisfying the conditions on (\ref{conditions-Spin}), the stratum $M_X (\Spin_{2n} ; \phi^{-1} (P) , s)^\circ$ is of dimension $\frac{1}{2}n(3n - 1)(g - 1) + s$.

\begin{remark}
In the above situation, there are two possibilities for the parabolic $P$ up to conjugacy. For: The orthogonal Grassmann variety parameterising $n$-dimensional isotropic subspaces of $\C^{2n}$ is a disjoint union $\OG (n, \C^{2n})_1 \sqcup \OG (n, \C^{2n})_2$. If for $i \in \{ 1, 2 \}$ we take $P_i$ to be the stabiliser of an element $\Lambda_i \in \OG (n, \C^{2n})_i$, then $P_1$ and $P_2$ are not conjugate in $\SO_{2n}$.

Suppose now that $E \to X$ is an $\SO_{2n}$-bundle and $\sigma \colon X \to E/P_1$ a reduction of structure group. This means that $E$ is given with respect to an appropriate covering $\uU$ by a cocycle $\{ \gamma_{\alpha \beta} \}$ where $\gamma_{\alpha \beta} \in P_1$. Recall now the sequence of nonabelian cohomology sets
\[
\cdots \ \to \ H^0 (X, \bZ_2 ) \ \to \ H^1 (X, \SO_{2n} ) \ \to \ H^1 (X, \Or_{2n} ) \ \to \ \cdots
\]
By \cite{Gro}, two $\SO_{2n}$-bundles $E$ and $E'$ map to the same $\Or_{2n}$-bundle if and only if they are exchanged by the group $H^0 (X, \Or_{2n} / \SO_{2n} ) = \bZ_2$ acting by conjugation by $h \in \Or_{2n} \setminus \SO_{2n}$; explicitly,
\[
h \cdot \{ \gamma_{\alpha \beta} \} \ = \ \{ h \gamma_{\alpha \beta} h^{-1} \} .
\]
Thus the bundle $h \cdot E$ admits a reduction of structure group to $h P_1 h^{-1}$. This is precisely the stabiliser of $h \Lambda_1$. It is not hard to show that $h \Lambda_1 \in \OG(n, \C^{2n})_2$, the opposite component, and so $h P_1 h^{-1}$ is conjugate over $\SO_{2n}$ to $P_2$ and not to $P_1$. In summary, the action of $H^0 (X, \Or_{2n} / \SO_{2n} )$ sends bundles with $P_1$-reductions to bundles with $P_2$-reductions.

Therefore, the loci $M_X ( \SO_{2n}, 0 ; P_1 , s ) $ and $M_X ( \SO_{2n}, 0 ; P_2 , s )$, which a priori are not expected to coincide in $M_X (\SO_{2n})$, are in fact identified in $M_X ( \Or_{2n}, 0 )$ and thereby in $M_X (\SL_{2n})$. This is in agreement with the irreducibility of
\[
\{ V \in M_X (\SL_{2n}) : V \hbox{ orthogonal and } t(V) \le t \}
\]
proven in \cite{CH2}.
\end{remark}

\begin{remark}
Using \cite[Theorem 3.1 and Theorem 5.1]{CH3}, one can write down a similar statement relating the stratification on $M_X (\Spin_{2n+1})$ with the one on $M_X ( \SO_{2n+1} , 0)$.   
\end{remark}

\section{Segre stratification for the Borel subgroup of \texorpdfstring{$\GL_3$}{GL\_3}}\label{sec:Borel}

Set $G = \GL_3$, and let $\delta \in \pi_1 (G) = \bZ$ be a topological type. Let $B \subset G$ be the Borel subgroup of upper triangular matrices. In this section we study the Segre stratification of $M_X (G, \delta)$ into the strata $M_X ( G, \delta ; B, s )$, $s > 0$, which is the simplest example associated to a nonmaximal parabolic subgroup. We will see that, in comparison with the case of a maximal parabolic (cf.\ Theorem \ref{GLrCase}), the situation is more involved: for example, the strata are not irreducible. 

As in \cite{CH3} and elsewhere, we will work with associated vector bundles rather than $G$-bundles directly, and construct these as iterated extensions. 
Given a principal $G$-bundle $E$ of topological type $\delta$, we denote by $V$ the associated vector bundle $E( \C^3 )$ of rank three and degree $\delta$. We write $\pi \colon E/B \to X$ for the associated homogeneous space bundle. If $\sigma \colon X \to E/B$ is a reduction of structure group, we denote by $E_B$ the corresponding principal $B$-bundle (cf.\ $\S$\ref{sec:reductions}).

\subsection{Iterated extensions}\label{ssec:iterated_extensions}

We firstly discuss the degree $d_B (E_B)$ for a $B$-bundle $E_B$ as in Definition \ref{defn:degree}. By (\ref{XstarP}), the character group $X^* (B)$ coincides with $X^* (T) = \bZ^3$ where $T \subset B \subset G$ is the maximal torus of diagonal matrices. Thus $\Hom (X^* (B), \bZ) \cong \Hom ( X^* (T), \bZ )$. Elements of the latter are in bijection with triples $\ud := (d_1, d_2, d_3) \in \bZ^3$, \linebreak where $\ud \colon X^* (T) \ \to \ \bZ$ 
sends the character
\begin{equation} \label{charchi}
\chi \ := \ \begin{pmatrix} t_1  & 0 & 0 \\ 0 & t_2 & 0 \\ 0 & 0 &  t_3 \end{pmatrix} \ \mapsto \ t_1^{\ell_1} \cdot t_2^{\ell_2} \cdot t_3^{\ell_3}
\end{equation}
to
\begin{equation} \label{character}
\deg \left( t \mapsto \chi \begin{pmatrix} t^{d_1}  & 0 & 0 \\ 0 & t^{d_2} & 0 \\ 0 & 0 &  t^{d_3} \end{pmatrix} \right) = \ell_1 d_1 + \ell_2 d_2 + \ell_3 d_3 \; .
\end{equation}

\begin{lemma} \label{lemma:structure}
Let $E \to X$ be a principal $G$-bundle of topological type $\delta \in \pi_1 (G) = \bZ$, and let $\sigma \colon X \to E/B$ be a reduction.

\begin{enumerate}
\renewcommand{\labelenumi}{(\alph{enumi})}
\item There exist line bundles $L_1$, $L_2$, $L_3$ together with rank two bundles $F$ and $F'$ over $X$ fitting into an exact and commutative diagram 
\begin{equation} \label{diagram:iteratedext} \xymatrix{
L_1 \ar[r] \ar[d]^\wr & F \ar[d] \ar[r] & L_2 \ar[d] \\
L_1 \ar[r] & V \ar[r] \ar[d] & F' \ar[d] \\
 & L_3 \ar[r]^\sim & L_3 \\
} \end{equation}
where $\deg L_1 + \deg L_2 + \deg L_3 = \delta$.
\item Write $d_i := \deg L_i$ for $1 \le i \le 3$. Then under the identification above, $d_B \left( E_{B} \right) = (d_1, d_2, d_3 ) \in \bZ^3$.
\item There is an exact sequence $0 \to L_1^{-1} L_2 \to \sigma^* T_\pi \to F^* \otimes L_3 \to 0$. In particular, $\deg \left(\sigma^*T_{\pi}\right)=2(d_3-d_1)$ and $s_B (E) \le 2 ( d_3 - d_1 )$.
\item If $E$ is stable, then $d_1 < \frac{\delta}{3} < d_3$ and $h^0 (X, \Hom (L_3, F)) = h^0 (X, \Hom (F', L_1 )) = 0$. Moreover, at least one of $F$ and $F'$ is a nontrivial extension.
\end{enumerate}
\end{lemma}

\begin{proof}
(a) By hypothesis, with respect to a suitable set of trivialisations over an open cover $\{ U_\alpha \}$ of $X$, the transition functions of $E$ are upper triangular, of the form
\begin{equation*}
\begin{pmatrix} \left( a_1 \right)_{\alpha \beta} & \star & \star \\ 0 & \left( a_2 \right)_{\alpha \beta} & \star \\ 0 & 0 & \left( a_3 \right)_{\alpha \beta}\end{pmatrix}\; .
\end{equation*}
The existence of the desired diagram follows immediately.

(b) Notice that the line bundle $L_i$ is determined by the cocycle $\left\{ (a_i)_{\alpha \beta} \right\}$. If $\chi \in X^* (B)$ is as in (\ref{charchi}),
 then $\langle d_B \left( E_{B} \right) , \chi \rangle$ is, by definition, the degree of the line bundle determined by the cocycle
\[
\left\{ \left( a_1 \right)_{\alpha \beta}^{\ell_1} \cdot \left( a_2 \right)_{\alpha \beta}^{\ell_2} \cdot \left( a_3 \right)_{\alpha \beta}^{\ell_3} \right\} ,
\]
which is $\ell_1 d_1 + \ell_2 d_2 + \ell_3 d_3$. Thus $d_B \left( E_{B} \right) = (d_1, d_2, d_3)$ in view of (\ref{character}).

(c) The full flag variety $G/B$ parametrises flags $0 \subset \mathsf{L} \subset \mathsf{F} \subset \C^3$. Globalising the universal flag over $G/B$, over $E/B$ we have a universal flag of subbundles $0 \subset \cL \subset \cF \subset \pi^* V$. Then $\sigma^* \cL = L_1$ and $\sigma^* \cF = F$. Now by Lie theory, $T_\pi$ is an extension
\[
0 \ \to \ \Hom \left( \cL , \frac{\cF}{\cL} \right) \ \to \ T_\pi \ \to \ \Hom \left( \cF , \frac{\pi^* V}{\cF} \right) \ \to \ 0 \;.
\]
Pulling back by $\sigma$, we obtain the desired exact sequence. For the rest: Noting that the dual $F^*$ is an extension $0 \to L_2^{-1} \to F^* \to L_1^{-1}$, we compute that $\det (\sigma^* T_\pi ) \cong L_1^{-2} L_3^2$. Therefore,
\[
s_B (E) \ \le \ \deg  \left( \sigma^* T_\pi \right) \ = \ 2 (d_3 - d_1 ) \;.
\]

(d) In view of (\ref{sPsndef}), the vector bundle $V$ is stable if and only if $E$ is a stable $G$-bundle. In this case, $d_1 < \frac{\delta}{3} < d_3$ since $L_1$ and $L_3$ are a subbundle and a quotient of $V$ respectively.

Next: If $\phi \colon L_3 \to F$ is a nonzero map, then the composition $V \to L_3 \overset{\phi}{\to} F \to V$ is an endomorphism which is not a homothety, and $V$ cannot be stable. The vanishing of $h^0 (X, \Hom (F', L_1) )$ is proven similarly.

Finally, if both $F$ and $F'$ are trivial extensions then there is a nonzero map
\[
L_2 \ \to \ F \ \to \ V \ \to \ F' \ \to \ L_2
\]
and $L_2$ is a direct summand of $V$, contradicting stability.
\end{proof}

\subsection{Nonemptiness of strata}

We give now a pair of lemmas which will be useful both for proving the existence of stable $G$-bundles and for deciding which Segre strata lie in the closures of which others. We continue to write $d_i$ for $\deg L_i$ in what follows.

\begin{lemma} \label{lemma:closure}
Let $0 \to L_1 \to F \to L_2 \to 0$ be an extension of line bundles. Then there exists an irreducible family $\cF \to T \times X$ containing $F$ and whose general element $\cF_t$ is an extension $0 \to M_t \to \cF_t \to N_t \to 0$ where $\deg M_t = d_1 - 1$ and $\deg N_t = d_2 + 1$. 
\end{lemma}

\begin{proof}
Suppose firstly that $d_2 - d_1 \ge g - 2$. Recall the Segre invariant $s_1$ as defined in (\ref{sndefn}). For any $F'$ of degree $d_1 + d_2$ we have
\[
(d_1 + d_2) - 2(d_1 - 1) \ \ge \ g \ \ge \ s_1 (F')
\]
in view of Theorem \ref{GLrCase}. Thus a general $F'$ of rank two and degree $d_1 + d_2$ has line subbundles of degree $d_1 - 1$, and any sufficiently general deformation of $F$ over an irreducible base will suffice. Suppose, then, that $d_2 - d_1 < g - 2$.

For any line bundle $P$ of sufficiently low degree, we can write $F$ as an extension of line bundles $0 \to P \to F \to Q \to 0$ with class $\varepsilon \in \PP H^1 (X, \Hom (Q, P)) \cong |\Kx Q P^{-1}|^*$, the latter identification by Serre duality. Further reducing $\deg P$ if necessary, we can assume that $P \ne L_1$, and therefore that $L_1 = Q(-D)$ for some effective divisor $D$ of degree $e := \deg Q - d_1$. By the proof of \cite[Proposition 1.1]{LN}, the class $\varepsilon$ belongs to the secant spanned by $D$ in $|\Kx Q P^{-1}|^*$.

\textbf{Claim:} $\Sec^e (X)$ is not dense in $|\Kx Q P^{-1}|^*$.

To see this, note firstly that $\dim \Sec^e (X) \le 2e - 1$. Thus it will suffice to show that $2e - 1 < \dim | \Kx Q P^{-1} |^*$. As $e = \deg Q - d_1$, by Riemann--Roch, this inequality would follow from
\[
2 \deg Q - 2d_1 - 1 \ < \ \deg Q - \deg P + g - 2 \; ,
\]
that is, $\deg Q + \deg P < 2d_1 + g - 1$. But
\[
\deg Q + \deg P \ = \ \deg F \ = \ d_1 + d_2 \; ,
\]
and by the hypothesis $d_2 - d_1 < g - 2$ this is certainly less than $2d_1 + g - 1$. This proves the Claim.

Noting that $\Sec^e (X)$ is contained in the closure of $\Sec^{e + 1} (X)$, by the Claim we can choose an irreducible curve $T$ in $|\Kx Q P^{-1}|^*$ containing $\varepsilon$ and of which a general element belongs to $\Sec^{e + 1} (X) \setminus \Sec^e (X)$. Again using the argument of \cite[Proposition 1.1]{LN}, we see that a general $t \in T$ defines an extension $0 \to P \to \cF_t \to Q \to 0$ to which some subsheaf $Q(-D'_t)$ lifts to a subbundle of $\cF_t$, where $\deg D'_t = e + 1=\deg Q-d_1+1$. Setting $M_t := Q(-D'_t)$ and $N_t := P(D'_t)$, we obtain the statement. (Note that the above argument does not require $0 \to L_1 \to F \to L_2 \to 0$ to be a nontrivial extension.)
\end{proof}

\begin{lemma} \label{lemma:deformations}
Suppose $E$ is a stable $G$-bundle admitting a $B$-reduction of degree $(d_1, d_2, d_3)$. Then there exist deformations of $E$ over irreducible bases $\tT$ and $\tT'$ where a general member of $\tT$ (resp., $\tT'$) is a stable $G$-bundle admitting a $B$-reduction of degree $( d_1 - 1 , d_2 + 1 , d_3 )$ (resp., $( d_1 , d_2 - 1 , d_3 + 1 )$).
\end{lemma}

\begin{proof}
As before, set $V = E(\C^3)$.
 By Lemma \ref{lemma:structure} (a), we have the diagram (\ref{diagram:iteratedext}) where $\deg L_i = d_i$. Let $\cF \to T \times X$ be the family constructed in Lemma \ref{lemma:closure}. As $E$ is stable, $h^0 (X, L_3^{-1} \otimes F ) = 0$ by Lemma \ref{lemma:structure} (d). Therefore, shrinking $T$ if necessary, we may assume that $h^1 (X, L_3^{-1} \otimes \cF_t )$ is constant on $T$. Let $\tT \to T$ be the projective bundle with fibre $\PP H^1 (X, L_3^{-1} \otimes \cF_t )$ at $t$. This is irreducible since $T$ is. In a standard way, $\tT$ parameterises nontrivial extensions $0 \to \cF_t \to V' \to L_3 \to 0$. By Lemma \ref{lemma:closure}, a general element of $\tT$ defines a diagram
\[ \xymatrix{
M_1 \ar[r] \ar[d]^\wr & \cF_t \ar[r] \ar[d] & M_2 \ar[d] \\
M_1 \ar[r] & V' \ar[r] \ar[d] & V/M_1 \ar[d] \\
 & L_3 \ar[r]^\sim & L_3
} \]
where $\deg M_1 = d_1 - 1$ and $\deg M_2 = d_2 + 1$. As $V$ is a stable bundle represented in $\tT$, stable bundles form a nonempty open subset of $\tT$.

For the rest: We observe that the dual $V^*$ has topological type $-\delta$ and admits a $B$-reduction of type $(-d_3, -d_2, -d_1 )$. Applying the above construction to $V^*$, we obtain a deformation whose general element $W$ is stable and admits a $B$-reduction of type \linebreak $(-d_3 - 1, -d_2 + 1, -d_1 )$. Dualising, we obtain a deformation of $V$ of which a general element admits a $B$-reduction of type $(d_1, d_2 - 1, d_3 + 1)$, as desired.
\end{proof}

\begin{proposition} \label{prop:existence}
For any tuple $\ud=(d_1, d_2, d_3)$ satisfying $d_1 < \frac{\delta}{3} < d_3$, there exist stable $G$-bundles admitting $B$-reductions of type $\ud$.
\end{proposition}

\begin{proof}
Set $d_0 := \left\lfloor \frac{\delta}{3} \right\rfloor$. Let us prove the existence for
\begin{equation} \label{EasyTypes}
\ud \ \in \ \left\{ (d_0 - 1 , d_0 , d_0 + 1) , (d_0 , d_0 , d_0 + 1) , (d_0 , d_0 + 1, d_0 + 1) \right\} ;
\end{equation}
note that these cover all possibilities for $\delta \;  (\!\!\mod 3)$. By duality, the second and third cases are equivalent. Furthermore, as stability is invariant under tensor product by line bundles, we may assume for convenience that $\ud = (-1, 0, 1)$ or $(0, 0, 1)$.

Let, then, $0 \to L_1 \to F \to L_2 \to 0$ and $0 \to F \to V \to L_3 \to 0$ be nontrivial extensions where $(d_1, d_2, d_3) = (-1, 0, 1)$, and consider a subbundle $W \subset V$. Then we have a diagram
\[ \xymatrix{
0 \ar[r] & F \ar[r] & V \ar[r] & L_3 \ar[r] & 0 \\
0 \ar[r] & W_1 \ar[r] \ar[u] & W \ar[r] \ar[u] & W_2 \ar[r] \ar[u] & 0
} \]
It is easy to see that $W$ destabilises $V$ if and only if either
\begin{enumerate}
\item[(i)] $W_2 = L_3$ and $W_1$ is a line bundle of degree $-1$, or
\item[(ii)] $W_1 = 0$ and $W \cong W_2$ is a lifting of a subsheaf $L_3 (-x)$ of $L_3$ for some $x \in X$.
\end{enumerate}
Suppose that (i) occurs. Then the map $F \to F / W_1$ factorises via $V$. By \cite[Lemma 3.2]{NR}, this is equivalent to
\[
\delta(V) \ \in \ \Ker \left( H^1 (X, \Hom (L_3, F)) \ \to \ H^1 (X, \Hom (L_3, F/W_1)) \right) ,
\]
where $\delta$ here is the connecting homomorphism in the long exact sequence. This kernel is a proper subspace of $H^1 (X, \Hom (L_3, F))$ because the target is nonzero and the map is surjective. Moreover, by \cite[Proposition 4.2]{LN} there are only finitely many possibilities for $W_1$. Thus a general $V$ is not destabilised by a $W$ of type (i).

For (ii): By \cite[Theorem 4.4 (i)]{CH0}, a lifting $L_3 (-x) \to V$ exists only if $\delta (V)$ belongs to the image of the ruled surface $\PP ( L_3^{-1} \otimes F )$ in $\PP H^1 (X, \Hom (L_3, F))$. Counting dimensions with Riemann--Roch, we see that this does not apply to a general $\delta (V)$.

Existence of stable $B$-bundles of type $(0, 0, 1)$ follows from a similar but easier argument, which we omit here.

With the existence established of stable bundles of the types in (\ref{EasyTypes}), as stability is open in families, by repeatedly applying the two constructions in Lemma \ref{lemma:deformations}, we can obtain stable $G$-bundles of topological type $\delta$ admitting $B$-reductions of type $(d_0 - k , d_0 + k - \ell , d_0 + \ell )$ for any $k, \ell \ge 0$. Recalling that $d_0 = \left\lfloor \frac{\delta}{3} \right\rfloor$, it is easy to check that any $\ud$ satisfying $d_1 + d_2 + d_3 = \delta$ and $d_1 < \frac{\delta}{3} < d_3$ is of this form for a unique choice of $k$ and $\ell$. This completes the proof of the proposition.
\end{proof}

\begin{remark}
Proposition \ref{prop:existence} shows in particular that there exist stable $G$-bundles $E$ with $s_B (E) = 2$ if $\delta \not\equiv 0 \;(\!\!\mod 3)$ and $s_B (E) = 4$ if $\delta \equiv 0 \;(\!\!\mod 3)$.

More generally, for each $\ud \in \pi_1 (T) = \bZ^3$ with $d_1 < \frac{\delta}{3} < d_3$, the locus
\begin{equation} \label{SegreTypeLocus}
M(\ud) \ := \ \left\{ E \in M_X (G, \delta) : E \hbox{ admits a $B$-reduction of type } \ud \right\} .
\end{equation}
is nonempty. By Lemma \ref{lemma:structure} (c), for any such $\ud$ we have $s_B (E) \le 2 (d_3 - d_1)$ for all $E \in M(\ud)$. However, in general we do not have $M(\ud) \subseteq M_X (G, \delta ; B, 2 (d_3 - d_1))^\circ$, as $E$ may admit a $B$-reduction of type $(d_1', d_2', d_3')$ where $2 (d_3' - d_1') < 2 (d_3 - d_1)$. Indeed, the difference $d_3 - d_1$ may be arbitrarily high, whereas by \cite[Theorem 1.1]{HN} we have $s_B (E) \le \dim (G/B) \cdot g = 3g$ for any $G$-bundle $E$. We return to the question of an upper bound on $s_B (E)$ in $\S$ \ref{sect:Hirschowitz}.
\end{remark}

\subsection{Geometry of strata}

We are now in a position to construct irreducible components of some Segre strata $M_X (G, \delta ; B, s)$ and prove Theorem \ref{thm:mainC} (i). 

\begin{theorem} \label{thm:irr}
Fix $\delta \in \pi_1 (G)$ and $\ud$ satisfying $d_1 < \frac{\delta}{3} < d_3$. Then the locus $M (\ud)$ defined in (\ref{SegreTypeLocus}) is nonempty and irreducible of dimension at most $6g - 5 + s$, where $s = 2 (d_3 - d_1 )$.
\end{theorem}

\begin{proof}
As $d_1 < \frac{\delta}{3} < d_3$, by Proposition \ref{prop:existence} the locus $M( \ud )$ is nonempty. To obtain the other statements, we construct a parameter space for bundles in $M (\ud)$.

Suppose firstly that $d_1 < d_2$. Then $h^0 (X, \Hom (L_2 , L_1)) = 0$ for all pairs $(L_1, L_2) \in \Pic^{d_1} (X) \times \Pic^{d_2} (X)$. By Riemann--Roch, the direct image sheaf over $\Pic^{d_1} (X) \times \Pic^{d_2} (X)$ with fibre $H^1 (X, \Hom (L_2 , L_1))$ at $(L_1, L_2)$ is locally free of rank $d_2 - d_1 + (g-1)$. The associated vector bundle $\bH_{d_1, d_2}\to \Pic^{d_1} (X) \times \Pic^{d_2} (X)$ is an irreducible variety of dimension $3g + d_2 - d_1 - 1$ which parameterises all extensions $0 \to L_1 \to F \to L_2 \to 0$ for $L_1$ and $L_2$ of degree $d_1$ and $d_2$ respectively.

Next, consider the open set
\[
\Omega \ := \ \left\{ ( F, L_3 ) \in \bH_{d_1 , d_2} \times \Pic^{d_3} (X) : h^0 (X, \Hom (L_3, F)) = 0 \right\} .
\]
We claim that $\Omega$ is nonempty. To see this: By an argument using \cite[Proposition 1.1]{LNa} and the hypothesis $d_2 > d_1$, a general element $F \in \bH_{d_1 , d_2}$ is a stable vector bundle. 

Moreover, as $d_3 > \frac{\delta}{3}$ by hypothesis, also $\mu (F) = \frac{d_1 + d_2}{2} < d_3 = \mu (L_3)$. 
Thus there are no nonzero maps $L_3 \to F$ for a general $F$ in any fibre of $\bH_{d_1, d_2}|_{(L_1 , L_2)}$.

It now follows from Riemann--Roch that the direct image sheaf over $\Omega$ with fibre \linebreak $H^1 (X, \Hom (L_3, F))$ at $(F, L_3 )$ is locally free of rank $2 d_3 - d_1 - d_2 + 2 (g-1)$. Let $\bB_\ud \to \Omega$ be the associated vector bundle; this is an irreducible variety of dimension
\begin{multline*}
\dim \bH_{d_1, d_2} + \dim \Pic^{d_3} + \rank \bB_\ud \ = \\
\left( 3g + d_2 - d_1 - 1 \right) + g + \left( 2 d_3 - d_1 - d_2 + 2 (g-1) \right) \ = \ 6g - 3 + 2 (d_3 - d_1 ) \; ,
\end{multline*}
which parameterises all extensions $0 \to F \to V \to L_3 \to 0$ where $(F, L_3) \in \Omega$, and thereby diagrams of the form (\ref{diagram:iteratedext}).

If $d_1 \ge d_2$ then we may no longer have $h^0 (X, \Hom (L_2, L_1) ) = 0$. However, in this case we have $d_2 < d_3$. Then by a very similar procedure to the above we may construct an analogous irreducible variety $\bB_\ud$ of dimension $6g - 3 + 2 ( d_3 - d_1 )$ parameterising diagrams of the form (\ref{diagram:iteratedext}) for $d_2 < d_3$.

Let $\Phi_\ud \colon \bB_\ud \dashrightarrow M_X (G, \delta)$ be the classifying map, which by Proposition \ref{prop:existence} is defined on a nonempty open subset. 
By Lemma \ref{lemma:structure} (a), any stable $G$-bundle $E$ of topological type $\delta$ with a $B$-reduction of type $\ud$ fits into a diagram of the form (\ref{diagram:iteratedext}) where $h^0 (X, \Hom (L_3, F)) = 0 = h^0 (X, \Hom (F', L_1))$. By construction, $V=E(\C^3)$ is represented in $\bB_\ud$. It follows that $M(\ud)$ in (\ref{SegreTypeLocus}) is precisely the image of $\bB_\ud$ by $\Phi_\ud$.

As for the dimension: The fibres of $\Phi_\ud \colon \bB_\ud \dashrightarrow M_X (G, \delta)$ are of dimension at least $2$, because proportional extension classes in $H^1 (X, L_2^{-1} L_1 ) = \bH_{d_1, d_2}|_{(L_1 , L_2)}$ define isomorphic extensions $0 \to L_1 \to F \to L_2 \to 0$ by \cite[Lemma 3.3]{NR}, and similarly for $H^1 (X, \Hom (L_3, F)) = \cH|_{(F, L_3)}$. Thus the image of $\Phi_\ud$ has dimension at most $\dim \bB_\ud - 2 = 6g - 5 + 2 (d_3 - d_1)$.
\end{proof}

Computing the dimension of (\ref{SegreTypeLocus}) can be delicate. Rather than attempting to treat this exhaustively, we give some examples. We continue to denote by $\Phi_\ud \colon \bB_\ud \dashrightarrow M_X (G, \delta)$ the moduli map defined in Theorem \ref{thm:irr}.

\begin{proposition} \label{prop:GenFin}
Let $\ud$ satisfy $d_1 < \frac{\delta}{3} < d_3$ and $d_j - d_i \le g - 1$ for $1 \le i \le j \le 3$. Then the moduli map $\Phi_\ud$ has generic fibre of dimension $2$, and $\dim M (\ud) = 6g - 5 + 2 (d_3 - d_1)$. Moreover, for a general $E \in \Image ( \Phi_\ud )$, we have $s_B (E) = 2 (d_3 - d_1)$.
\end{proposition}

\begin{proof}
Firstly, we define the open set
\begin{multline*}
\Pi \ := \ \left\{ ( L_1 , L_2 , L_3 ) \in \Pic^{d_1} (X) \times \Pic^{d_2} (X) \times \Pic^{d_3} (X) : \right. \\
 \left. h^0 (X, L_i^{-1} L_j ) = 0 \hbox{ for } 1 \le i \le j \le 3 \right\} .
\end{multline*}
By hypothesis, we may assume that $\Pi$ is nonempty, and therefore dense in $\Pic^{d_1} (X) \times \Pic^{d_2} (X) \times \Pic^{d_3} (X)$. Consider the composed map
\[
\rho \colon \bB_\ud \ \to \ \bH_{d_1, d_2} \times \Pic^{d_3} (X) \ \to \ \Pic^{d_1} (X) \times \Pic^{d_2} (X) \times \Pic^{d_3} (X) .
\]
It will suffice to show that the restriction of $\Phi_\ud$ to the dense open subset $\rho^{-1} (\Pi)$ has fibres of dimension $2$.

Let $E$ be the $G$-bundle arising from an iterated extension belonging to $\rho^{-1} (\Pi )$. Write as usual $\pi \colon E/B \to X$ for the homogeneous space bundle and $\sigma \colon X \to E/B$ for the $B$-reduction of type $\ud$ arising from the iterated extension structure. Using Lemma \ref{lemma:structure} (c), we obtain a diagram
\[ \xymatrix{
 & & L_2^{-1} L_3 \ar[d] \\
L_1^{-1} L_2 \ar[r] & \sigma^* T_\pi \ar[r] & F^* \otimes L_3 \ar[d] \\
 & & L_1^{-1} L_3
} \]
By definition of $\Pi$ we have $h^0 (X, \sigma^* T_\pi ) = 0$. Denote by $N_\sigma$ the normal bundle of $\sigma (X)$ in $E/B$. Then $h^0 (X, N_\sigma) = 0$ by Lemma \ref{lemma:normal}.

Now $H^0 (X, N_\sigma)$ is the tangent space to the Hilbert scheme of sections $X \to E/B$ at the point corresponding to the image of $\sigma$. As this is zero, there are only finitely many $B$-reductions $\sigma_1 \colon X \to E/B$ of type $\ud$ where $(L_1, L_2, L_3) \in \Pi$; and $\Phi_\ud^{-1} (E)$ intersects $\rho^{-1} ( L_1, L_2, L_3 )$ for only finitely many $(L_1, L_2, L_3) \in \Pi$. For each such $\sigma_1$, the isomorphism classes of $F$ and $F'$ are determined as $\sigma_1^* \cF$ and $\sigma_1^* (\pi^* V/ \cL)$ as in the proof of Lemma \ref{lemma:structure} (c).

Now the $L_i$ are simple bundles and $h^0 (X, L_1^{-1} L_2) = 0$ by hypothesis. Thus by \cite[Lemma 3.3]{NR} the locus of extension classes in $H^1 (X, L_2^{-1} L_1) = \bH_{d_1, d_2}|_{(L_1 , L_2)}$ representing $F$ is of dimension at most $1$. Similarly, if $F$ is stable then in the same way the locus in $H^1 (X, L_3^{-1} \otimes F) = \bB_\ud|_{(F, L_3)}$ of extensions isomorphic to $V$ is of dimension at most $1$. Thus a general fibre of $\Phi_\ud|_\Pi$ has dimension $2$, and we obtain the dimension statement.

Lastly: By Theorem \ref{thm:irr}, the locus of bundles $E'$ with $s_B (E') = s' < s$ is of dimension at most $6g - 5 + s' \le 6g - 5 + s - 2$. Therefore, a general $E$ in the image of $\Phi_\ud$ must have $s_B (E) = 2 ( d_3 - d_1 )$.
\end{proof}

Proposition \ref{prop:GenFin} shows that, for a tuple $\ud$ satisfying $d_1 < \frac{\delta}{3} < d_3$ and $d_j - d_i \le g - 1$ for $1 \le i \le j \le 3$, a general bundle in $\Image(\Phi_{\ud})$ admits a reduction in $M(\ud)$ and, indeed, that reduction achieves the Segre value $s=s_B(E)=2(d_3-d_1)$.

\begin{remark} \label{BadBeh}
In cases where the numerical conditions in Proposition \ref{prop:GenFin} are not satisfied, the image of $\Phi_\ud$ may in fact belong to $M_X (G, \delta ; B, s')$ for $s' < 2 (d_3 - d_1):=s$. Suppose that $d_2 - d_1 \ge g + 1$. Each element of $\bB_\ud$ is an iterated extension
\[ \xymatrix{
L_1 \ar[r] & F \ar[r] \ar[d] & L_2 \\
 & V \ar[d] & \\
& L_3 , &
} \]
whence $s_B (E) \le s$ by Lemma \ref{lemma:structure} (c). But by the Hirschowitz bound \cite{Hir} or by Theorem \ref{GLrCase}, the Segre invariant $s_1 (F)$ defined in (\ref{sndefn}) satisfies $s_1 (F) \le g$. Thus $F$ has a line subbundle $M$ such that $\deg F - 2 \deg M \le g$; in other words,
\[
\deg M \ \ge \ \frac{1}{2} ( d_1 + d_2 - g  ) \ \ge \ \frac{1}{2} ( d_1 + (d_1 + g + 1) - g ) \ = \ \frac{1}{2} (2d_1 + 1 ) ;
\]
that is, $\deg M \ge d_1 + 1$. Therefore, $V$ is also an iterated extension of the form
\[ \xymatrix{
M \ar[r] & F \ar[r] \ar[d] & F/M \\
 & V \ar[d] & \\
& L_3 &
} \]
By Lemma \ref{lemma:structure} (c), therefore, $s_B (E) \le 2 (d_3 - (d_1 - 1)) \le s-2$. We conclude that in fact 
\[\Image \left( \Phi_\ud \right) \ \subseteq \ M_X (G, \delta ; B, s-2).\]

We observe that in this case the fibres of $\Phi_\ud$ must have dimension at least $4$, as \linebreak $\dim M_X (G, \delta; B, s-2 ) \le 6g - 7 + s$ by Theorem \ref{thm:irr}.
\end{remark}

We have now another application of Lemma \ref{lemma:deformations}, giving information on which of the loci $M(\ud')$ contain bundles specialising to bundles in $M(\ud)$ and proving Theorem \ref{thm:mainC} (ii). This generalizes the last sentence of Theorem \ref{GLrCase} for vector bundles. 

\begin{corollary}
The locus $M(d_1, d_2, d_3)$ belongs to
\begin{equation} \label{TwoTypes}
\overline{M (d_1 - 1, d_2 + 1, d_3)} \cap \overline{M (d_1, d_2 - 1, d_3 + 1)} .
\end{equation}
In particular, $M_X (G, \delta ; B, s)^\circ \subseteq \overline{M_X (G, \delta ; B, s + 2)}$.
\end{corollary}

\begin{proof}
Suppose that $E \in M (\ud)$. By Lemma \ref{lemma:deformations}, there exists a deformation of $E$ of which a general element admits a $B$-reduction of type $(d_1 - 1, d_2 + 1, d_3)$. Thus $E$ belongs to the closure of the irreducible locus $M (d_1 - 1, d_2 + 1, d_3)$. In the same way, $E \in \overline{M(d_1, d_2 - 1, d_3 + 1)}$.
\end{proof}

\subsection{A Hirschowitz bound for \texorpdfstring{$\GL_3$}{GL\_3}-bundles} \label{sect:Hirschowitz}

We continue to abbreviate $\GL_3$ to $G$, and write $B$ for the Borel subgroup of upper triangular invertible matrices.

By \cite[Theorem 1.1]{HN} for any reductive group $H$ and parabolic $P \subset H$, one has $s_P (E) \le g \cdot \dim (H/P)$. This shows in particular that $s_B (E) \le 3g$ for any $G$-bundle $E$. The upper bound $g \cdot \dim (G/P)$ has been sharpened in case $P$ is a maximal parabolic in $\GL_r$ in \cite{Hir} (see also \cite{CH0}) and for certain parabolics in $\Sp_{2n}$ and $\SO_r$ in \cite{CH2, CH3}. Here we use Proposition \ref{prop:GenFin} to obtain a sharp upper bound for $s_B$ in a similar way. 

As the computations become rather tedious in general, we content ourselves with proving the bound (Theorem \ref{thm:mainC} (iii)) under convenient numerical hypotheses on $\delta$ and $g$.

\begin{corollary}[Hirschowitz bound] \label{cor:Hirschowitz}
Suppose that $\delta \equiv 0 \;(\!\!\mod 3)$ and $g \equiv 1 \;(\!\!\mod 6)$. Then for any $E \in M_X (G, \delta)$ we have $s_B (E) \le 3(g-1)$.
\end{corollary}

\begin{proof}
Firstly, by Theorem \ref{thm:irr}, we have $\dim M (\ud) \le 6g - 5 + s$ where $s = 2(d_3 - d_1)$. Thus $M (\ud)$ can be dense in $M_X (G, \delta)$ only if $6g - 5 + s \ge 9g - 8=\dim M_X (G, \delta)$; that is, $s \ge 3 (g-1)$. Let us show that one $M(\ud)$ is indeed dense for the value $s=3(g-1)$. 

By hypothesis, $g = 6u + 1$ for some $u \ge 1$. We may take any tuple of the form $(d_1, d_2, d_3) \ = \ (-5u+k, u+k, 4u+k)$, so that
\begin{enumerate}
\item[(i)] $d_2 - d_1 = 6u = g-1$, and
\item[(ii)] $d_3 - d_1 = 9u = \frac{3(g - 1)}{2}$
\end{enumerate}
As $g = 6u + 1$ is odd and $d_1 + d_2 = -4u + 2k$ is even, by Theorem \ref{GLrCase} a general $F$ of degree $d_1 + d_2$ satisfies $s_1 (F) = g-1$. As $d_2 - d_1 = g-1$ by (i), a general extension $0 \to L_1 \to F \to L_2 \to 0$ represented in $\bH_{d_1, d_2}$ may be taken to be general in moduli. If $L_3$ is also general, then by Hirschowitz's Lemma $F^* \otimes L_3$ may be assumed to be Brill--Noether general. Therefore, applying (i) and (ii) we obtain
\begin{multline*}
h^0 (X, F^* \otimes L_3 ) \ = \ \chi (X, F^* \otimes L_3 ) \ = \ 2 d_3 - d_2 - d_1 - 2(g - 1) \ = \\
2 (d_3 - d_1) - (d_2 - d_1) - 2 (g-1) \ = \ 3 (g-1) - (g-1) - 2(g-1) \ = \ 0 .
\end{multline*}
By (i), also $h^0 (X, L_1^{-1} L_2 ) = 0$ for general $L_1$ and $L_2$. As $N_\sigma$ is an extension \linebreak $0 \to L_1^{-1} L_2 \to N_\sigma \to F^* \otimes L_3 \to 0$, we have $h^0 (X, N_\sigma ) = 0$. Arguing as in the proof of Proposition \ref{prop:GenFin}, we see that a general fibre of $\Phi_\ud \colon \bB_\ud \dashrightarrow M_X (G, \delta)$ has dimension $2$, and the image has dimension
\[
\dim \bB_\ud - 2 \ = 6g-3+2(d_3-d_1)-2=\ 6g - 5 + 3(g - 1) \ = \ 9g - 8 \ = \ \dim M_X (G, \delta) \;.
\]
In particular, $\Phi_\ud$ is dominant, and $s_B (E) \le 3(g-1)$ for a general $E \in M_X (G, 0)$. By the semicontinuity proven in Theorem \ref{semicont}, the same is true for all $G$-bundles $E$ of topological type $\delta$.
\end{proof}

We conclude by giving as complete an overview as possible in a particular case.  

\begin{example}
Set $\delta = 0$ and $g = 7$. In the following diagram\footnote{Diagram produced using GeoGebra, \texttt{www.geogebra.org}.}, the horizontal and vertical axes represent $d_1$ and $d_3$ respectively, and $d_2 = -d_1 - d_3$.

\begin{center}

\includegraphics[scale=0.4]{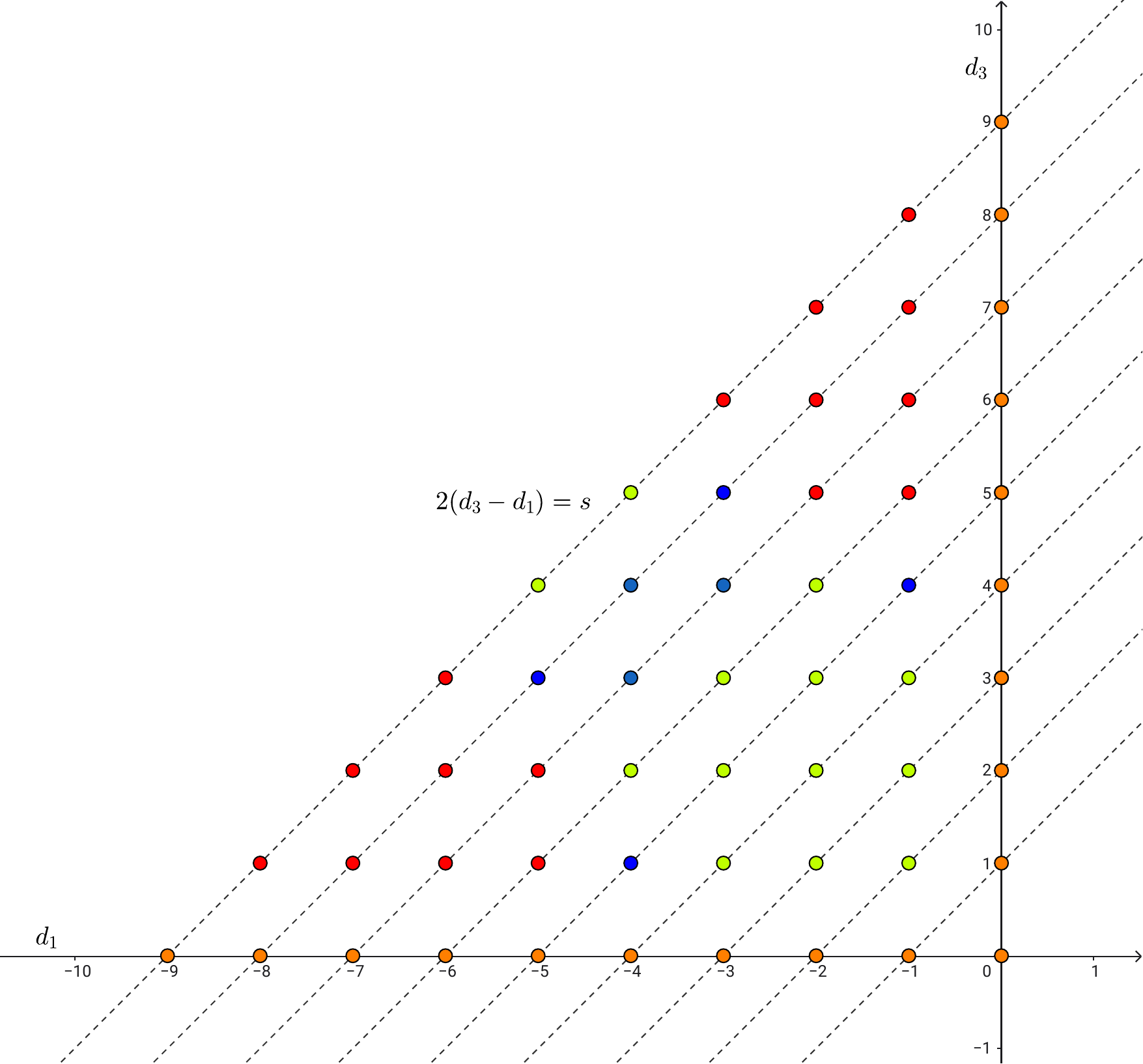}

\end{center}

\noindent By Theorem \ref{thm:irr}, all integral points in the interior of the second quadrant give rise to $\ud$ such that $M (\ud)$ is nonempty. The dotted lines parameterise those $\ud$ which in addition satisfy $2(d_3 - d_1) = s$; we show these ``level sets'' for $0 < \frac{s}{2} \le \frac{3}{2} (g-1) = 9$, in view of the Hirschowitz bound in Corollary \ref{cor:Hirschowitz}.

Now since $\delta = 0$, the operation of dualising vector bundles (corresponding to extending the structure group to $\GL_3$ by $\gamma \mapsto {^t\gamma^{-1}}$) sends $M(d_1, d_2, d_3)$ isomorphically to $M(-d_3, -d_2, -d_1)$. Combining this observation with the results above, we have:
\begin{itemize}
\item[\textcolor{green}{\textbullet}] The green dots represent those $\ud$ for which $M(\ud)$ attains the maximal dimension $6g - 5 + s$ by Proposition \ref{prop:GenFin} or Corollary \ref{cor:Hirschowitz}, and then $M(\ud)$ is generically contained in $M_X (G, 0; B, 2(d_3-d_1))^{\circ}$. In particular, both irreducible components $M(-5, 1, 4)$ and $M(-4, -1, 5)$ are dense in the moduli space $M_X (G, 0)$, and so a general $G$-bundle $E$ admits $B$-reductions of both types.
\item[\textcolor{red}{\textbullet}] The red dots represent those $\ud$ discussed in Remark \ref{BadBeh}, having \linebreak $\dim M(\ud)\leq 6g-7+s < 6g - 5 + s$, with $s=2(d_3-d_1)$ and belonging to a stratum $M_X (G, 0; B, s' )$ for $s' < s$. For example, this is the case for the tuple $\ud=(-8,7,1)$. 
\item[\textcolor{blue}{\textbullet}] The blue dots represent those types $\ud$ which are not covered by the aforementioned results, and whose dimension and location remain to be ascertained. 
\item[\textcolor{orange}{\textbullet}] Finally, the orange dots represent those $\ud$ for which a $G$-bundle admitting a $B$-reduction of type $\ud$ is, at best, strictly semistable.
\end{itemize}
\end{example}

\end{document}